\documentclass[11pt,a4paper,abstract=on]{scrartcl}

\usepackage[english]{babel}

\usepackage{lmodern}
\usepackage[T1]{fontenc}

\usepackage{ucs}
\usepackage[utf8x]{inputenc}

\usepackage[comma,numbers,sort&compress]{natbib}
\usepackage{amsfonts,amstext,amsmath,amssymb,amsopn,amsthm}

\usepackage{mathtools}

\usepackage{dsfont}
\usepackage{mathrsfs}

\usepackage{listings}
\usepackage{physics} 
\usepackage[dvipsnames, svgnames, x11names]{xcolor}

\usepackage{siunitx}
\usepackage{graphicx}
\usepackage{subfig}
\usepackage{wrapfig}

\usepackage{csquotes}
\usepackage[colorlinks=true, 
linkcolor=blue,
pdfstartview=FitH,      
breaklinks=true,        
bookmarksopen=true,     
bookmarksnumbered=true  
]{hyperref}
\usepackage[capitalize,nameinlink]{cleveref}
\usepackage{url}
\usepackage{enumitem}
\usepackage{caption}

\usepackage[top=1in, bottom=1.5in, left=1in, right=1in]{geometry}

\usepackage{algorithm,float}
\usepackage[noend]{algpseudocode}
\algrenewcommand\algorithmicrequire{\textbf{Initialization:}}

\allowdisplaybreaks

\renewcommand{\equiv}{\ensuremath{:=}}

\newcommand{\act}[1]{\left\langle {#1} \right\rangle}

\newcommand{\mymap}[3]{#1:\,#2 \to #3\,}

\newcommand{\mysetc}[2]{\left\{#1\,\left|\,#2\right.\right\}}

\newcommand{\cex}[2]{\ensuremath{\mathbb{E}[#1\,|\,#2]}}
\newcommand{\cpr}[2]{\ensuremath{\mathbb{P}(#1\,|\,#2)}}
\newcommand{\1}{\ensuremath{\mathds{1}} }
\newcommand{\icol}[1]{
  \left(\begin{smallmatrix}#1\end{smallmatrix}\right)%
}

\newcommand{\wto}{\xrightarrow{\textrm{w}}}
\newcommand{\cb}{\ensuremath{\overline{\mathbb{B}}}}

\DeclareMathOperator{\supp}{supp}
\DeclareMathOperator{\dist}{dist}

\DeclareMathOperator*{\argmin}{\arg\!\min}

\DeclareMathOperator{\prox}{prox}

\DeclareMathOperator{\id}{Id}

\DeclareMathOperator{\Fix}{Fix}
\DeclareMathOperator{\cl}{cl}

\DeclareMathOperator{\Hess}{Hess}

\makeatletter
\def\@endtheorem{\endtrivlist\@endpefalse }
\makeatother

\newtheorem{thm}{Theorem}[section]
\newtheorem{cor}[thm]{Corollary}
\newtheorem{lemma}[thm]{Lemma}
\newtheorem{prop}[thm]{Proposition}
\theoremstyle{definition}
\newtheorem{example}[thm]{Example}
\newtheorem{definition}[thm]{Definition}

\newtheoremstyle{note}
{3pt}
{3pt}
{}
{}
{\bfseries}
{\bfseries :}
{.5em}
{}
\theoremstyle{note}

\newtheorem{rem}[thm]{Remark}

\title{Random Function Iterations for Consistent
Stochastic Feasibility}
\author{Neal Hermer\thanks{Institut f\"ur Numerische und Angewandte Mathematik,
Universit\"at G\"ottingen,
37083 G\"ottingen, Germany. NH was supported by 
Deutsche Forschungsgemeinschaft Research Training Grant 2088 TP-B5.
E-mail:  \texttt{n.hermer@math.uni-goettingen.de}}, 
D. Russell Luke\thanks{Institut f\"ur Numerische und Angewandte Mathematik,
Universit\"at G\"ottingen,
37083 G\"ottingen, Germany. DRL was supported in part by 
Deutsche Forschungsgemeinschaft Research Training Grant 2088 TP-B5.
E-mail:  \texttt{r.luke@math.uni-goettingen.de}} 
and  Anja Sturm\thanks{Institut f\"ur Mathematische Stochastik,
Universit\"at G\"ottingen,
37077 G\"ottingen, Germany. AS was supported in part by Deutsche Forschungsgemeinschaft 
Research Training Grant 2088 TP-B5.
E-mail:  \texttt{asturm@math.uni-goettingen.de}}}
\date{September 19, 2018}

\begin{document}
\maketitle

\begin{abstract}
  We study the convergence of iterated random functions for stochastic
  feasibility in the consistent case (in the sense of Butnariu and
  Flåm \cite{ButnariuFlam1995}) in several different settings, under
  decreasingly restrictive regularity assumptions of the fixed point
  mappings.  The iterations are Markov chains and, for the purposes of
  this study, convergence is understood in very restrictive
  terms. We show that sufficient conditions for geometric (linear)
  convergence in expectation of stochastic projection algorithms
  presented in Nedi\'c \cite{Nedic2011}, are in fact {\em necessary}
  for geometric (linear) convergence in expectation more generally of iterated
  random functions.
\end{abstract}

{\small \noindent {\bfseries 2010 Mathematics Subject Classification:}
  {Primary
    60J05, 
    52A22, 
    49J55 
    Secondary 49J53, 
    65K05.\\ 
  }}


\section{Introduction}
\label{sec:introduction}

We are inspired by problems of the form
\begin{align*}
  \text{Find } x^{*} \in \bigcap_{i \in I} C_{i}.
\end{align*}
where $I$ is an index set and $C_i$ are closed subsets of a metric
space.  When $I$ is a finite set, then this is the classical {\em
  deterministic feasibility problem}.  Randomized algorithms for
solving such finite intersections have been intensely studied in
recent years for their application to distributed computational
schemes and machine learning.  Our study here concerns a
generalization where $I$ is an arbitrary (possibly uncountable) set.
This was first considered by Butnariu and Fl\r{a}m
\cite{ButnariuFlam1995,flam1995,Butnariu1995}, where this is called
the {\em stochastic feasibility problem}.  There are many reasons why
one might consider such infinite feasibility problems.  At the time
Butnariu and Fl\r{a}m's work appeared, there was a great deal of
interest in solution methods for linear integral equations.  Our
primary motivation is to propose a new way of modeling and analyzing
errors, either numerical or measurement, as they are manifest in
numerical iterative procedures.

The simplest algorithm one could imagine for solving such problems is
to generate sequences $(x_{k})_{k\in\mathbb{N}}$ by the fixed point
iteration
\[
x_{k+1}\in \left(\prod_{i\in I} P_{C_{i}}\right)x_k
\]
where $P_{C_i}$ is the metric projection onto the set $C_{i}$ and
$\prod_{i\in I}$ indicates the composition of the operators over the
index set. This is known as the cyclic projections algorithm and has
been extensively studied for the case when $I$ is a finite set.  But
when $I$ is infinite one immediately encounters the problem that such
an algorithm never completes the first iteration!

One application where infinite feasibility problems appear very
naturally is integral equations of the first kind in the separable
Hilbert space $L_{2}([a,b])$, as considered in \cite{Butnariu1995}:
\begin{align}\label{e:ie}
  (Tx)(t) = \int_{a}^{b} K(t,s) x(s) \dd{s} = g(t) \qquad t \in [a,b],
\end{align}
with $g \in L_{2}([a,b])$.  The feasibility reformulation of this
problem is
\begin{align}\label{e:ie-feas}
  \mbox{ Find } x\in \bigcap_{t\in[a,b],~ a.s.}
  C_t\equiv\left\{\varphi\in
    L_{2}([a,b])~|~\left(T\varphi\right)(t)=g(t) \right\}.
\end{align}
The {\em almost surely (a.s)} under the intersection will be clarified
below.  The basic idea, however, is to choose the parameter $t$ above
{\em randomly} and to choose the nearest point in the set $C_t$ to the
current guess.  As such, the sequences that we generate are random
processes.  We return to this application in Section \ref{sec:appl}.

One of our main contributions is to place the previous results on
stochastic projections iterations in the broader context of random
function iterations, that is, iterations of randomly selected
mappings with arbitrary initial distributions from which the initial
points are chosen.  Following \cite{Nedic2010} we characterize the
sequence generated by the method of Random Function Iterations
(\cref{algo:SSM}) as a Markov chain (\cref{thm:SSMasMC}).  There are
many different notions of convergence of Markov chains.  For {\em
  consistent} feasibility problems considered in this study (Standing
Assumption 2) almost sure convergence applies, which perhaps is of
limited interest in the broader context of generic Markov chains.
When we move to {\em inconsistent} stochastic feasibility problems,
what we formulate more generally as stochastic fixed point problems in
a follow-up study, the richness of the theory of Markov processes
comes much more into play.  For the consistent case we are able to
establish convergence results in a number of new settings, namely {\em
  compact metric spaces} or $\mathbb{R}^{n}$ with {\em
  paracontracting} mappings (see Section \ref{sec:compact metr} and
\ref{sec:findimnormvc}, resp.) and separable Hilbert spaces with {\em
  averaged} mappings
(see Section \ref{sec:analysisSSM}).  (It is the technology of
averaged mappings that opens the door to an analysis of algorithms for
inconsistent stochastic feasibility problems and more generally
stochastic fixed point problems.)

Our framework allows us to address a variety of Markov chains and
stochastic algorithms, though to fix these ideas our main example will
be stochastic sequential projection algorithms.  We also achieve with
this analysis more refined convergence statements about the
corresponding sequences of random variables, namely almost surely
strong and, under certain assumptions, geometric convergence (called
{\em linear} or exponential in other communities) of the measures
\cref{thm:LinCVGMeasure}.  We then show {\em necessary} and sufficient
conditions for geometric convergence in expectation
\cref{thm:Equiv_lin_cvg} in a stochastic analog to \cite[Theorem 3.12
and Corollary 3.13]{Luke2016b}.  These conditions are identified as 
a manifestation of {\em metric subregularity} of a suitable 
merit function at its zeros.   When specialized to convex
feasibility, this formulation of metric subregularity of the merit function is 
equivalent to the function possessing what is known as the KL property 
\cref{thm:equiv_linConv_propR}.   Finally, we identify a previously
unrecognized {\em necessary} condition for geometric convergence of
stochastic sequential projection algorithms, \cref{thm:PropCond}.  Despite the
strong assumptions of the present study, a number of unexpected
behaviors can occur; these are demonstrated in concrete examples.

\section{Stochastic Fixed Point Theory}
\label{sec:consistent_feas_prob}

Consider a collection of continuous mappings $\mymap{T_{i}}{G}{G}$, $i
\in I$, on a metric space $(G,d)$, where $I$ is an arbitrary index
set. Assume that $(I,\mathcal{I})$ is a measurable space.  Let
$(\Omega,\mathcal{F},\mathbb{P})$ be a probability space. Let $\xi$ be
an $I$-valued random variable, i.e.\ $\mymap{\xi}{\Omega}{I}$
measurable ($\xi^{-1} A \in \mathcal{F}$ for all $A \in \mathcal{I}$).

Let $(\xi_{n})_{n\in\mathbb{N}}$ be an iid sequence with $\xi_{n}
\overset{\text{d}}{=} \xi$.  Let $\mu$ be a probability measure on
$(G,\mathcal{B}(G))$. The stochastic selection method is given by
\begin{algorithm}[H]
  \caption{Random Function Iteration (RFI)}\label{algo:SSM}
  \begin{algorithmic}[0]
    \Require{$X_{0} \sim \mu$} \For{$k=0,1,2,\ldots$} \State{ $
      X_{k+1} = T_{\xi_{k}} X_{k}$}
    \EndFor
    \State \Return $\{X_{k}\}_{k \in \mathbb{N}}$
  \end{algorithmic}
\end{algorithm}
\vspace*{-\baselineskip} \medskip Here we mean by $X_{0} \sim \mu$,
that the \emph{law} (i.e.\ distribution) of $X_{0}$ satisfies
$\mathcal{L}(X_{0}) := \mathbb{P}^{X_{0}} := \mathbb{P}(X_{0} \in
\cdot) := \mathbb{P} \circ X_{0}^{-1} = \mu$. The following
assumptions will be employed throughout.

\paragraph{Standing Assumption 1}
\label{ass:1}
\begin{enumerate}[label=\alph*)]
\item\label{item:ass1:indep} $X_{0},\xi_{0},\xi_{1}, \ldots, \xi_{k}$
  are independent for every $k \in \mathbb{N}_{0}$, where
  $(\xi_{k})_{k \in \mathbb{N}}$ are i.i.d.
\item\label{item:ass1:Phi} The function $\mymap{\Phi}{G\times I}{G}$,
  $(x,i)\mapsto T_{i}x$ is measurable.
\end{enumerate}

\subsection{RFI as a Markov chain}
\label{sec:SPMasMC}

The iterates of the RFI $X_{k+1} = \Phi(X_{k},\xi_{k}) =
T_{\xi_k}X_k$, $k \in \mathbb{N}_{0}$, can be considered as a time
homogeneous Markov chain with transition kernel
\begin{equation}\label{eq:trans kernel}
  (x\in G) (A\in
  \mathcal{B}(G)) \qquad p(x,A) \equiv \mathbb{P}(\Phi(x,\xi) \in A) =
  \mathbb{P}(T_{\xi}x \in A)   
\end{equation}
where $\Phi$ is called an \emph{update function}.

To see that $p$ is really a transition kernel, recall that, in
general, a transition kernel $\mymap{p}{G\times
  \mathcal{B}(G)}{[0,1]}$ is measurable in the first argument, i.e.\
$p(\cdot,A)$ is measurable for all $A \in \mathcal{B}(G)$ (follows
from \cite[Lemma 1.26]{kallenberg1997}) and is a probability measure
in the second argument, i.e.\ $p(x,\cdot)$ is a probability measure
for all $x \in G$ (immediate by definition).  Now recall the
definition of a Markov chain with general transition kernel $p$.
\begin{definition}
  A sequence of random variables $(X_{k})_{k \in \mathbb{N}_{0}}$,
  $\mymap{X_{k}}{(\Omega,\mathcal{F},\mathbb{P})}{(G,\mathcal{B}(G))}$
  is called Markov chain with transition kernel $p$ if for all $k \in
  \mathbb{N}_{0}$ and $A \in \mathcal{B}(G)$ $\mathbb{P}$-a.s.\ holds
  \begin{enumerate}[label=(\roman*)]
  \item $\cpr{X_{k+1} \in A}{X_{0}, X_{1}, \ldots, X_{k}} =
    \cpr{X_{k+1} \in A}{X_{k}}$
  \item $\cpr{X_{k+1} \in A}{X_{k}} = p(X_{k},A)$.
  \end{enumerate}
\end{definition}
Since $(G,\mathcal{B}(G))$ is a Borel space and $X_{k}$ is a random
  variable in $G$, it makes sense to talk about these conditional
  expectations (existence of regular versions of the conditional
  distribution by \cite[Theorem 5.3]{kallenberg1997}).

From the setting above, the following fact follows easily from
\cref{thm:disinteg}.
\begin{prop}\label{thm:SSMasMC}
  Under \nameref{ass:1}, the sequence of random variables $(X_{k})_{k
    \in \mathbb{N}_{0}}$ generated by \cref{algo:SSM} is a Markov
  chain with transition kernel $p$ given by \eqref{eq:trans kernel}.
\end{prop}

Let $\mu\in \mathscr{P}(G)$, where $\mathscr{P}(G)$ is the space of
all probability measures on $G$.  The Markov operator $\mathcal{P}$
acting on a measure $\mu$ is defined via
\begin{align*}
  (A \in \mathcal{B}(G))\qquad \mu\mathcal{P} (A) := \int_{G} p(x,A)
  \mu(\dd{x}).
\end{align*}
One defines the operation of the Markov operator acting on a
measurable function $\mymap{f}{G}{\mathbb{R}}$ via
\begin{align*}
  (x\in G)\qquad \mathcal{P}f(x):= \int_{G} f(y) p(x,\dd{y}).
\end{align*}
Note that
\begin{align*}
  \mathcal{P}f(x) = \int_{G} f(y) \mathbb{P}^{\Phi(x,\xi)}(\dd{y}) =
  \int_{\Omega} f(\Phi(x,\xi(\omega))) \mathbb{P}(\dd{\omega})=
  \int_{I} f(\Phi(x,u)) \mathbb{P}^{\xi}(\dd{u}).
\end{align*}

There are several notions of convergence that one can study for the
sequence $(X_{k})$ on the metric space $(G,d)$, or correspondingly of
the law $(\mathcal{L}(X_{k}))$ on $\mathscr{P}(G)$, where
$\mathcal{L}(X_k)\equiv\mathbb{P}^{X_k}\equiv
\mathbb{P}(X_k\in\cdot)$.  Denote the set of bounded and continuous
functions from $G$ to $\mathbb{R}$ by $C_{b}(G)$.  Let $(\nu_{n})$ be
a sequence of probability measures on $G$.  The sequence $(\nu_{n})$
converges to $\nu$ in the {\em weak sense}, if $\nu \in
\mathscr{P}(G)$ and for all $f \in C_{b}(G)$ it holds that $\nu_{n} f
\to \nu f$ as $n \to \infty$.

A fixed point of the Markov operator $\mathcal{P}$ is called an
\emph{invariant distribution}, i.e.\ $\pi \in \mathscr{P}(G)$ is
invariant if and only if $\pi \mathcal{P} = \pi$.  A very general
notion of convergence in the context of Markov chains concerns the
convergence of the probability measures $\nu_k\equiv
\tfrac{1}{k}\sum_{j=1}^k\mathcal{L}(X_j)$ in the weak sense.  An
elementary fact from the theory of Markov chains
(\cref{thm:construction_inv_meas}) is that, if this converges, it
converges to an invariant probability measure $\pi$: For $f \in
C_{b}(G)$,
\begin{align*}
  \nu_k f= \mathbb{E}\left[\frac{1}{k}\sum_{j=1}^{k} f(X_{j})\right]
  \to \pi f, \qquad \text{as } k \to \infty.
\end{align*}
The notion of convergence we consider here is much stronger; we
consider almost sure convergence of the sequence $(X_k)$ to a random
variable $X$:
\[
X_k\to X \text{ a.s.}, \qquad \text{as } k \to \infty.
\]
  
Clearly, almost sure convergence of the sequence implies the more
general notion above.  This is common in the studies of stochastic
algorithms in optimization, though this does not require the full
power of the theory of general Markov processes.  For {\em consistent
  feasibility} (defined below) however, this is all that is needed for
our present purposes.  In a follow-up study we will need to consider
more general notions of convergence of this Markov chain.

\subsection{Consistent Stochastic Feasibility Problem}\label{sec:consist SSM}
The {\em stochastic feasibility} problem is to find a point
\begin{align}
  \label{eq:stoch_feas_probl}
  x^{*} \in C := \mysetc{x \in G}{\mathbb{P}(x \in \Fix T_{\xi}) = 1},
\end{align}
where the fixed point set of the operator $T_{i}$ is denoted as
\begin{align*}
  \Fix T_{i} = \mysetc{x \in G}{x = T_{i}x}.
\end{align*}
We assume throughout that not only is $\Fix T_{i}$ non-empty for $\mathbb{P}^{\xi}$-almost all  $i \in I$, but
more restrictively: 
\vspace{-\baselineskip}\paragraph{Standing Assumption 2 (consistent feasibility problem).}
\label{item:ass1:Cnonemp} The set $C$ is nonempty.
\medskip

\noindent Note that due to continuity of $T_{i}$ it follows, that
$\Fix T_{i}$ is a closed set.  This specializes immediately to the
stochastic feasibility problem formulated by Butnariu and Fl\r{a}m
\cite{ButnariuFlam1995} where $\Fix T_{\xi}=C_{\xi}$.  In order to
make sense of the specialization to stochastic set feasibility, we
need the event $\{x \in \Fix T_{\xi}\}$ to be an element of
$\mathcal{F}$ for any $x\in G$.
\begin{rem}\label{rem:feas_set_measurable}
  Since $\{x\}\in \mathcal{B}(G)$ and the function
  $\mymap{\Phi_{\xi}}{G\times \Omega}{G}$, $(x,\omega)\mapsto
  \Phi_{\xi}(x,\omega) := (\Phi \circ (\id,\xi))(x,\omega)=
  T_{\xi(\omega)}x$ is measurable as composition of two measurable
  functions, we find
  \begin{align*}
    \{x \in \Fix T_{\xi}\} &= \mysetc{\omega \in \Omega}{x \in \Fix
      T_{\xi(\omega)}} \\ &= \mysetc{\omega \in
      \Omega}{T_{\xi(\omega)}x = x} \\ &= \mysetc{\omega \in
      \Omega}{(x,\omega) \in \Phi_{\xi}^{-1}\{x\}} \\ & \in
    \mathcal{F},
  \end{align*}
  since slices of sets in the product $\sigma$-field are measurable
  with respect to the single $\sigma$-fields (see
  \cref{lemma:slices_prod_field}).
\end{rem}

\noindent Denote in the following for $A\subset\Omega$,
$C(A):=\bigcap_{\omega \in A} \Fix T_{\xi(\omega)}$.
\begin{lemma}[equivalence of stochastic and deterministic feasibility 
  problems]\label{lemma:equiv_SCFP_CFP}
  Under the standing assumptions, and if $G$ is complete and
  separable, there exists a $\mathbb{P}$-nullset $N\subset\Omega$,
  such that
  \begin{align*}
    C= C(\Omega\setminus N) = \bigcap_{\omega \in \Omega \setminus N}
    \Fix T_{\xi(\omega)}.
  \end{align*}
  Furthermore, $C \subset G$ is closed.
\end{lemma}
\begin{proof}
  For the direction "$\supset$", note that $\mathbb{P}(\Omega\setminus
  N) = 1$ for any $\mathbb{P}$-nullset $N \subset \Omega$, so for $x
  \in C(\Omega\setminus N)$ it holds that $\mathbb{P}(x\in \Fix
  T_{\xi}) = 1$, i.e.\ $x \in C$.\\ Consider now the direction
  "$\subset$".  Let $Q$ be a dense and countable subset of $C$ (exists
  by \cref{thm:dense_sets_separable_space}). Since for each $q\in Q$,
  $\mathbb{P}(q\in \Fix T_{\xi}) = 1$, there is $N_{q} \subset\Omega$
  with $\mathbb{P}(N_{q}) = 0$ and $q \in C(\Omega\setminus
  N_{q})$. Set $N = \bigcup_{q\in Q}N_{q}$, then $\mathbb{P}(N)=0$ and
  $q\in C(\Omega\setminus N)$ for all $q\in Q$.\\
  Now let $c \in C$, so $\exists (q_{n})_{n\in\mathbb{N}} \subset Q$
  with $q_{n} \to c$ as $n \to \infty$. Since, for all $i\in I$, $\Fix
  T_{i}$ is closed by continuity of $T_{i}$, we get $c = \lim_{n\to
    \infty} q_{n} \in C(\Omega\setminus N)$.\\
  The set $C(\Omega\setminus N)$ is defined as intersection over
  closed sets and hence closed itself.
\end{proof}
\begin{rem}[interpretation]
  \cref{lemma:equiv_SCFP_CFP} shows that the feasible set $C$ in the
  separable case can be written as intersection of a selection of sets
  $\Fix T_{\xi(\omega)}$ as in the deterministic formulation of the
  fixed point problem, but where $\omega \in \Omega \setminus N$ for a
  nullset $N \subset \Omega$. In fact $C(\Omega)$ is in general a
  proper subset of $C = C(\Omega\setminus N)$ or can even be
  empty. But note that, even though the 
construction of $C$ in \cref{lemma:equiv_SCFP_CFP} appears to 
depend on the random variable $\xi$, in fact $C$ only depends on the 
{\em distribution} $\mathbb{P}^{\xi}$ by definition.
  Furthermore, in the context of more general Markov chains, we have,
  \[
  (c\in C)\qquad p(c, \{c\})=\mathbb{P}(T_{\xi}c\in\{c\})=
  \mathbb{P}(\Omega\setminus N)=1.
  \]
  Hence
  \[
  (A\in\mathcal{B}(G))\qquad \delta_c\mathcal{P}(A)=p(c,
  A)=\1_A(c)=\delta_c(A).
  \]
  In other words, the delta function $\delta_c$ for $c\in C$ is an
  invariant measure for $\mathcal{P}$.
\end{rem}

\begin{cor}[$\mathbb{P}^{\xi}$ nullset, separable
  space] \label{cor:Pxi_nullset} Under the assumptions of
  \cref{lemma:equiv_SCFP_CFP} there exists a $\mathbb{P}$-nullset $N$
  with $C=C(\Omega\setminus N)$, such that $\xi(N) := \mysetc{
    \xi(\omega)}{ \omega \in N}$ is a $\mathbb{P}^{\xi}$-nullset,
  where we denote $\mathbb{P}^{\xi} = \mathbb{P}(\xi \in \cdot)$, and
  it satisfies
  \begin{align*}
    C = \bigcap_{ i \in \xi(\Omega) \setminus \xi(N)} \Fix T_{i}.
  \end{align*}
\end{cor}
\begin{proof}
  We will construct a $\mathbb{P}$-nullset $N$ for which
  $\xi(\Omega\setminus N) = \xi(\Omega) \setminus \xi(N)$, where
  $\xi(N)$ is a $\mathbb{P}^{\xi}$-nullset, in that case immediately
  follows that
  \begin{align*}
    \bigcap_{\omega \in \Omega \setminus N} \Fix T_{\xi(\omega)} =
    \bigcap_{ i \in \xi(\Omega) \setminus \xi(N)} \Fix T_{i}.
  \end{align*}
  Let $A_{x}:=\mysetc{i\in I}{T_{i}x = x}$ for $x \in G$, then
  analogously to \cref{rem:feas_set_measurable}
  \begin{align*}
    A_{x} = \mysetc{i\in I}{(x,i) \in \Phi^{-1} \{x\}} \in \mathcal{I}
  \end{align*}
  and so is $A := \bigcap_{c \in C} A_{c} = \bigcap_{q \in Q} A_{q}$
  as countable intersection of measurable sets ($Q \subset C$ dense
  and countable, see proof of \cref{lemma:equiv_SCFP_CFP}). Let
  $\tilde N$ be the $\mathbb{P}$-nullset from
  \cref{lemma:equiv_SCFP_CFP}, i.e.\ $C=C(\Omega\setminus\tilde N)$,
  note that due to
  \begin{align*}
    C = \bigcap_{\omega \in \Omega \setminus \tilde N} \Fix
    T_{\xi(\omega)} = \bigcap_{i \in \xi(\Omega\setminus\tilde N)}
    \Fix T_{i} \neq \emptyset
  \end{align*}
  it holds $\xi(\Omega\setminus \tilde N) \subset A_{c} \neq \emptyset
  $, for all $c \in C$. Set $N:= \Omega \setminus \xi^{-1}A$, then
  from $\Omega\setminus \tilde N \subset \xi^{-1}A$ follows $N \subset
  \tilde N$ is a $\mathbb{P}$-nullset and
  \begin{align*}
    \mathbb{P}^{\xi}(A) = \mathbb{P}(\xi^{-1}A) \ge
    \mathbb{P}(\Omega\setminus \tilde N) =1,
  \end{align*}
  i.e.\ $\mathbb{P}^{\xi}(\xi(N))=1-\mathbb{P}^{\xi}(A)=0$. By
  definition of $A$ we have for $\omega \in \xi^{-1}A$, that any $c
  \in C$ satisfies $c \in \Fix T_{\xi(\omega)}$, so it follows $C
  \subset C(\xi^{-1}A)$. Due to $C(\Omega \setminus N) \subset
  C(\Omega \setminus \tilde N)$ holds
  \begin{align*}
    C = \bigcap_{\omega \in \xi^{-1}A} \Fix T_{\xi(\omega)} =
    \bigcap_{i \in \xi(\xi^{-1} A)} \Fix T_{i} = \bigcap_{ i \in
      \xi(\Omega) \setminus \xi(N)} \Fix T_{i}.
  \end{align*}
  Note that from $N=\Omega\setminus\xi^{-1}A$ follows
  $\xi(N)=\xi(\Omega)\setminus\xi(\xi^{-1}A)$.
\end{proof}

  If $\xi$ is not surjective, then $\xi(\Omega) \neq I$. In that case,
  there is a $\mathbb{P}^{\xi}$-nullset $\xi(N)$ of indices in $I$,
  that are not needed to characterize the fixed point
  set, and these indices can be removed from the index set $I$.
  Note also that, in general, the $\mathbb{P}$-nullsets occurring in
  \cref{lemma:equiv_SCFP_CFP} and \cref{cor:Pxi_nullset} are
  different. If there is $N \subset \Omega$ with $C=C(\Omega\setminus
  N)$, then it need not be the case that $C= \bigcap_{i \in \xi(\Omega)\setminus
    \xi(N)} \Fix T_{i}$.

 In the context of the iterates $X_k$ of \cref{algo:SSM} in
many of the results below we construct the set $N$ in
\cref{lemma:equiv_SCFP_CFP} as follows:
\begin{equation}\label{eq:N}
  N = \bigcup_{k}N_{k}\quad\mbox{ where }N_k\equiv \Omega\setminus
  \{\omega\in\Omega~|~T_{\xi_{k}(\omega)}c=c~\forall c\in C \}.
\end{equation}
From \cref{lemma:equiv_SCFP_CFP} we have that $N_k$ is a set of
measure zero, hence so is $N$.

\section{Convergence Analysis}
\label{sec:conveergence}

We achieve convergence of iterated random functions for consistent
stochastic feasibility in several different settings under different
assumptions on the metric spaces and the mappings $T_i$ ($i\in I$).
The main properties of the mappings we consider are:
\begin{itemize}
\item quasi-nonexpansive mappings, i.e.
  \begin{align}\label{eq:qne}
    (\forall x\notin \Fix T_{i})(\forall y \in \Fix T_{i})\qquad
    d(T_{i}x,y) \le d(x,y).
  \end{align}
\item paracontractions, i.e.\ $T_{i}$ is continuous and
  \begin{align}\label{eq:pc}
    (\forall x\notin \Fix T_{i})(\forall y\in \Fix T_{i})\qquad
    d(T_{i}x,y) < d(x,y);
  \end{align}
\item nonexpansive mappings, i.e.
  \begin{align}\label{eq:ne}
    (\forall x,y \in G) \qquad d(T_{i} x, T_{i} y) \le d(x,y);
  \end{align}
\item averaged mappings on a normed linear space $\mathcal{H}$, i.e.\
  mappings $\mymap{T}{\mathcal{H}}{\mathcal{H}}$ for which there
  exists an $\alpha\in(0,1)$ such that
  \begin{align}\label{d:averaged}
    (\forall x,y\in\mathcal{H})\qquad \norm{T x - T y}^{2} +
    \frac{1-\alpha}{\alpha}\norm{ (x - Tx) - (y -T y)}^{2} \le
    \norm{x-y}^{2}.
  \end{align}
\end{itemize}
Note that for a quasi-nonexpansive mapping $\mymap{T}{G}{G}$ the
condition $x\in \Fix T$ implies that $d(T x,y) = d(x,y)$ for all $y
\in G$. The set of quasi-nonexpansive mappings contains the
paracontractions and the nonexpansive mappings. The set of projectors
onto convex sets or more generally the set of averaged mappings on a
Hilbert space $\mathcal{H}$ is contained in both the set of
nonexpansive mappings and the set of paracontractions \cite[Remark
4.24 and 4.26]{Bauschke2011}.  For an example of a paracontraction
that is not averaged see \cref{eg:Huber} and \cref{sec:paracont}.
Averaged mappings were first used in the work of Mann, Krasnoselski,
Edelstein, Gurin, Polyak and Raik who wrote seminal papers in the
analysis of (firmly) nonexpansive and averaged mappings
\cite{mann1953mean, krasnoselski1955, edelstein1966, Gubin67} although
the terminology ``averaged'' wasn't coined until sometime later
\cite{BaiBruRei78}.

\begin{example}\label{eg:Huber}
  Let $\mymap{f}{\mathbb{R}}{\mathbb{R}_{+}}$ be continuous. Let
  $f(0)=0$ and $\abs{f(x)} < \abs{x}$ for all $x \in \mathbb{R} \setminus
  \{0\}$, then $f$ is paracontractive. This includes also convex
  functions, e.g.\ Huber functions, which are not averaged in general
  (see \cref{sec:paracont}). For other examples on
  $\mathbb{R}^{n}$ also see \cref{sec:paracont}.
\end{example}

\subsection{RFI on a compact metric space}
\label{sec:compact metr}

In this section we establish convergence of the RFI on a compact
metric space.  The next example illustrates why nonexpansivity alone
does not suffice to guarantee convergence to the intersection set $C$.
\label{sec:SSMcompact}

\begin{example}[nonexpansive mappings, negative result]
  For non-expansive mappings in general, one cannot expect that the
  support of every invariant measure is contained in the feasible set
  $C$.  Consider a rotation in positive direction in $\mathbb{R}^{2}$
  \begin{align*}
    A=
    \begin{pmatrix}
      \cos(\varphi) &-\sin(\varphi)\\
      \sin(\varphi)&\cos(\varphi)\\
    \end{pmatrix}, \qquad \varphi \in (0,2 \pi),
  \end{align*}
  and set $\xi= 1$ and $I=\{1\}$, $T_{1}=A$.  Then $C=\{0\}$ and,
  since $\norm{A} = 1$, $A$ is nonexpansive, but $\norm{Ax}=\norm{x}$
  for all $x \in \mathbb{R}^{2}$. So the (deterministic) iteration
  $X_{k+1}=A X_{k}$ will not converge to $0$, whenever $X_{0} \sim
  \delta_{x}$, $x\neq 0$.
\end{example}

A sufficient requirement on the mappings $T_{i}$ to ensure convergence
of \cref{algo:SSM} is paracontractiveness. The next Lemma is the main
ingredient for proving a.s.\ convergence of $(X_{k})$ to a random
point in $C$.  The support of a probability measure $\nu \in
\mathscr{P}(G)$ is the smallest closed set $S \subset G$, for which
$\nu(G\setminus S)=0$ (see also \cref{thm:supp_measure} for equivalent
representations); we then write $S = \supp \nu$.

\begin{lemma}[invariant measures for
  para-contractions] \label{lemma:invMeasforParacontra} Under the
  standing assumptions and if $T_{i}\,(i\in I)$ is paracontracting on
  a compact metric space, then the set of invariant measures for
  $\mathcal{P}$ is $\mysetc{\pi \in \mathscr{P}(G)}{\supp \pi \subset
    C}$.
\end{lemma}
\begin{proof}
  It is clear that $\pi \in \mathscr{P}(G)$ with $\supp \pi \subset C$
  is invariant, since $p(x,\{x\}) = \mathbb{P}(T_{\xi}x \in \{x\}) =
  \mathbb{P}(x \in \Fix T_{\xi}) = 1$ for all $x \in C$ and hence
  $\pi\mathcal{P}(A) = \int_{C}p(x,A) \pi(\dd{x}) = \pi(A)$ for all $A
  \in \mathcal{B}(G)$.

  The other implication is not so immediate.  Suppose $ \supp \pi
  \setminus C \neq \emptyset$ for some invariant measure $\pi$ of
  $\mathcal{P}$.  Then due to compactness of $\supp \pi$ (as it is
  closed in $G$) we can find $s \in \supp \pi$ maximizing the
  continuous function $\dist(\cdot,C)$ on $G$. So $d_{\text{max}} =
  \dist(s,C) > 0$. We show that the probability mass around $s$ will
  be attracted to the feasible set $C$, implying that the invariant
  measure loses mass around $s$ in every step, which yields a
  contradiction.

  Define the set of points being more than $d_{\text{max}}-\epsilon$
  away from $C$:
  \begin{align*}
    K(\epsilon) := \mysetc{x \in G}{ \dist(x,C) > d_{\text{max}} -
      \epsilon }, \qquad \epsilon \in (0,d_{\text{max}}).
  \end{align*}
  This set is measurable, i.e.\ $K(\epsilon) \in \mathcal{B}(G)$,
  because it is open. Let $M(\epsilon)$ be the event in $\mathcal{F}$,
  where $T_{\xi}s$ is at least $\epsilon$ closer to $C$ than $s$,
  i.e.\
  \begin{align*}
    M(\epsilon) := \mysetc{ \omega \in \Omega}{ \dist(T_{\xi(\omega)}
      s,C) \le d_{\text{max}} - \epsilon}.
  \end{align*}
  There are two possibilities, either there is an $\epsilon \in
  (0,d_{\text{max}})$ with $\mathbb{P}(M(\epsilon)) >0$ or no such
  $\epsilon$ exists.  In the latter case we have $\dist(T_{\xi}s,C) =
  d_{\text{max}} =\dist(s,C)$ a.s.\ by paracontractiveness of
  $T_{i}$. By compactness of $C$ there exists $c \in C$ such that
  $0<d_{\text{max}} = d(s,c)$.  Hence the probability of the set of
  $\omega\in\Omega$ such that $s \not \in \Fix T_{\xi(\omega)}$ is
  positive and so is the probability that $\dist(T_{\xi(\omega)}s, C)
  \le d(T_{\xi(\omega)}s,c) < d(s,c)$ - a contradiction.

  So it must hold that there is an $\epsilon \in (0,d_{\text{max}})$
  with $\mathbb{P}(M(\epsilon)) >0$.  In view of continuity of the
  mappings $T_{i}$ around $s$, $i \in I$, define
  \begin{align*}
    A_{n} : = \mysetc{\omega \in M(\epsilon)}{ d(T_{\xi(\omega)}
      x,T_{\xi(\omega)} s) \le \tfrac{\epsilon}{2} \quad \forall
      x\in\mathbb{B}(s,\tfrac{1}{n})}\quad (n \in \mathbb{N}).
  \end{align*}
  It holds that $A_{n} \subset A_{n+1}$ and $\mathbb{P}(\bigcup_{n}
  A_{n}) = \mathbb{P}(M(\epsilon))$. So in particular there is an $m
  \in \mathbb{N}$, $m \ge 2/\epsilon$ with $\mathbb{P}(A_{m}) > 0$.
  For all $x \in \mathbb{B}(s,\tfrac{1}{m})$ and all $\omega\in A_{m}$
  we have
  \begin{align*}
    \dist(T_{\xi(\omega)} x, C) \le d(T_{\xi(\omega)}x,
    T_{\xi(\omega)} s) + \dist(T_{\xi(\omega)} s,C) \le d_{\text{max}}
    - \frac{\epsilon}{2},
  \end{align*}
  which means $T_{\xi(\omega)}x \in G \setminus
  K(\tfrac{\epsilon}{2})$.  Hence, in particular we conclude that
  \[
  p(x,K(\tfrac{\epsilon}{2})) < 1 \quad\forall x \in
  \mathbb{B}(s,\tfrac{1}{m}).
  \]
  Since $p(x,K(\epsilon)) = 0$ for $x \in G$ with $\dist(x,C) \le
  d_{\text{max}} - \epsilon$ due to paracontractiveness, it holds by
  invariance of $\pi$ that
  \begin{align*}
    \pi(K(\epsilon)) = \int_{G} p(x,K(\epsilon)) \pi(\dd{x}) =
    \int_{K(\epsilon)} p(x,K(\epsilon)) \pi(\dd{x}).
  \end{align*}
  It follows, then, that
  \begin{align*}
    \pi(K(\tfrac{\epsilon}{2})) &= \int_{K(\tfrac{\epsilon}{2})}
    p(x,K(\tfrac{\epsilon}{2})) \pi(\dd{x}) \\ &=
    \int_{\mathbb{B}(s,\tfrac{1}{ m})}p(x, K(\tfrac{ \epsilon}{2}) )
    \pi( \dd{x}) + \int_{K(\tfrac{\epsilon}{2}) \setminus
      \mathbb{B}(s,\tfrac{1}{ m})} p(x,K(\tfrac{\epsilon}{2}))
    \pi(\dd{x}) \\ &< \pi( \mathbb{B}(s ,\tfrac{1}{m})) +
    \pi(K(\tfrac{\epsilon}{2}) \setminus \mathbb{B}(s, \tfrac{1}{
      m}))= \pi(K( \tfrac{ \epsilon}{2}))
  \end{align*}
  which is a contradiction.  So the assumption that $\supp
  \pi\setminus C \neq \emptyset$ is false, i.e.\ $\supp \pi \subset C$
  as claimed.
\end{proof}

\begin{thm}[Theorem 4.22 in \cite{Hairer2006}]\label{thm:Feller}
  Under Standing Assumption 1, if $T_{i}\, (i\in I)$ is continuous,
  then the Markov operator $\mathcal{P}$ is Feller, i.e.\
  $\mathcal{P}: C_{b}(G)\to C_{b}(G)$.
\end{thm}
\begin{proof}
  By continuity of $T_{i}, \, i \in I$, the update function $\Phi$ is
  continuous in the first argument. It follows for $f \in C_{b}(G)$
  and $x_{n} \to x$ as $n\to \infty$ by Lebesgue's Dominated
  Convergence Theorem
  \begin{align*}
    \mathcal{P}f(x_{n}) = \int_{I} f(\Phi(x_{n},u))
    \mathbb{P}^{\xi}(\dd{u}) \to \int_{I} f(\Phi(x,u))
    \mathbb{P}^{\xi}(\dd{u}) = \mathcal{P}f(x).
  \end{align*}
  Note that $\mathcal{P}f$ is bounded, whenever $f$ is a bounded
  function.
\end{proof}

\begin{thm}[almost sure convergence for a compact metric
  space] \label{thm:asCVG_compMetricSpace} Under the standing
  assumptions, let $T_{i}$ be paracontractive, $i \in I$, and let
  $(G,d)$ be a compact metric space. Then the sequence $(X_{k})$ of
  random variables generated by \cref{algo:SSM} converges almost
  surely to a random variable $X_{\mu} \in C$ depending on the initial
  distribution $\mu$.
\end{thm}
\begin{proof}
  Since $\mathcal{P}$ is Feller and $G$ compact,
  \cref{thm:construction_inv_meas} implies that any subsequence of
  $(\nu_{n})$, where $\nu_{n} = \tfrac{1}{n} \sum_{i=1}^{n}
  \mathcal{L}(X_{i})$, has a convergent subsequence and clusterpoints
  are invariant measures for $\mathcal{P}$. Let $(\nu_{n_{k}})$ be a
  convergent subsequence with limit $\pi$. So for the bounded and
  continuous function $\dist(\cdot,C)$ it holds that $\nu_{n_{k}}
  \dist(\cdot,C) \to \pi \dist(\cdot,C) = 0$ as $k \to \infty$ by weak
  convergence of the probability measures and the
  fact that, by \cref{lemma:invMeasforParacontra}, $\supp \pi \subset C$.\\
  Due to quasi-nonexpansiveness and \cref{lemma:equiv_SCFP_CFP} (a
  compact metric space is separable), we have a.s.\ (for all
  $\omega\notin N$ with $N$ given by \eqref{eq:N}) that $d(X_{k+1},c)
  \le d(X_{k},c)$ for all $c\in C$ and $k\in \mathbb{N}$, which
  implies $\dist(X_{k+1},C) \le \dist(X_{k},C)$ for all $k \in
  \mathbb{N}$ a.s.  It therefore follows that
  \[
  \mathbb{E}[ \dist(X_{n_{k}},C)] \le \tfrac{1}{n_{k}}
  \sum_{i=1}^{n_{k}} \mathbb{E}[\dist(X_{i},C)] =
  \nu_{n_k}\dist(\cdot, C)\to 0
  \]
  by monotonicity of $(\mathbb{E}[\dist(X_{k},C)])_{k}$.  This yields
  $\mathbb{E}[\dist(X_{k},C)] \to 0$ as $k \to \infty$.  Now since
  $(\dist(X_{k},C))_{k}$ is nonincreasing, it must be that
  $\dist(X_{k},C) \to 0$ a.s.  Hence for any cluster point
  $x_{\omega}$ of $(X_{k}(\omega))_{k}$ we have $x_{\omega} \in C$.
  This together with a.s.\ monotonicity of $(d(X_{k},c))_{k}$ for all
  $c \in C$ implies that $d(X_{k}(\omega), x_{\omega}) \to 0$ for any
  cluster-point $x_{\omega}$ of $(X_{k}(\omega))_{k}$, which implies
  the uniqueness of $x_{\omega}$.  In other words, $(X_{k})$ converges
  almost surely to a random variable $X_{\mu}$, with
  $X_{\mu}(\omega)=x_{\omega} \in C$, $\omega \not \in N$, as claimed.
\end{proof}

\subsection{Finite dimensional normed vector space}
\label{sec:findimnormvc}
The results for compact metric spaces can be applied, with minor
adjustments, to finite dimensional vector spaces.  In the following
let $(G,d) = (V,\norm{\cdot})$ be a finite dimensional normed vector
space over $\mathbb{R}$. This means in particular, that $V$ is also
complete and every closed and bounded set is compact (Heine-Borel
property) and all norms on $V$ are equivalent. So actually, since all
$n$-dimensional vector spaces are isomorphic, it is enough to study
convergence in $\mathbb{R}^{n}$ equipped with the euclidean norm
$\norm{\cdot}$.

The following result for $\mathbb{R}^{n}$ is a straight forward
application of Theorem \ref{thm:asCVG_compMetricSpace}.
\begin{thm}[almost sure convergence in
  $\mathbb{R}^{n}$] \label{thm:asCvg_Rn} Under the standing
  assumptions, let $\mymap{T_{i}}{\mathbb{R}^{n}}{\mathbb{R}^{n}}$ be
  paracontractive, $i \in I$. Then the sequence $(X_{k})$ of random
  variables generated by \cref{algo:SSM} converges almost surely to a
  random variable $X_{\mu} \in C$ depending on the initial
  distribution $\mu$.
\end{thm}
\begin{proof}
  First, suppose $\mu=\delta_{x}$ for $x\in \mathbb{R}^n$. Let $N$ be
  given by \eqref{eq:N}. The quasi-nonexpansiveness property gives us
  $\norm{X_{k+1} - c} \le \norm{X_{k} - c}$ for all $c \in C$ a.s.\
  (i.e.\ if $\omega \not \in N$). Letting $c\in C$ with
  $\dist(x,C)=\norm{x-c}$, this implies $X_{k} \in \cb(c,\norm{x-c})$,
  where $\cb(s,\epsilon) \subset \mathbb{R}^{n}$ is the closed ball
  around $s \in \mathbb{R}^{n}$ with radius $\epsilon$. The assertion
  $X_{k} \to X_{\delta_{x}}$ a.s.\ then follows from
  \cref{thm:asCVG_compMetricSpace}. Denote the corresponding invariant
  measure as $\pi_{x} := \mathcal{L}(X_{\delta_{x}})$.
  
  Suppose now that $\mu \in \mathscr{P}(\mathbb{R}^{n})$ is arbitrary.
  For $f \in C_{b}(\mathbb{R}^{n})$ one has $p^{k}(x,f) \le
  \norm{f}_{\infty}$ for all $k\in \mathbb{N}$ and $x \in
  \mathbb{R}^{n}$.  Note that $p^{k}(x,f)=\delta_x\mathcal{P}^{k}f$,
  and from the above argument $\delta_x\mathcal{P}^{k}\to \pi_x$ in
  the weak sense as $k\to \infty$.  Hence by Lebesgue's Dominated
  Convergence Theorem, we get
  \[
  \mu \mathcal{P}^{k}f = \int_{\mathbb{R}^{n}} p^{k}(x,f) \mu(\dd{x})
  \to \int_{\mathbb{R}^{n}} \pi_{x}f \mu(\dd{x}) =: \mu \pi_{x} f =:
  \pi_{\mu} f, \qquad \text{as }k\to\infty.
  \]
  We conclude that $\mathcal{L}(X_{k}) = \mu \mathcal{P}^{k} \to
  \pi_{\mu}$ weakly.  The measure $\pi_{\mu} = \mu \pi_{x}$ is an
  invariant probability measure for $\mathcal{P}$, since $\pi_{x}$ is
  a invariant probability measure for $\mathcal{P}$.

  Choosing $f = \min\{\dist(\cdot,C),M\} \in C_{b}(\mathbb{R}^{n})$
  with $M > 0$ yields $\mathcal{L}(X_{k})f \to \pi_{\mu}f = 0$.  Since
  $f(X_{k+1}) \le f(X_{k})$ a.s.\ and $\mathcal{L}(X_{k})f =
  \mathbb{E}[f(X_{k})] \to 0$, it holds that $f(X_{k}) \to 0$ a.s. In
  particular, $\dist(X_{k},C) \to 0$ a.s.  So for a converging
  subsequence $(X_{n_{k}}(\omega))_{k}$ with limit $x_{\omega}$ it
  holds that $x_{\omega} \in C$.  Moreover, since
  $(\norm{X_{k}-x_{\omega}})_{k}$ is monotone, actually $X_{k}(\omega)
  \to X_{\mu}(\omega):= \lim_{k} X_{k}(\omega)= x_{\omega} \in C$,
  $\omega \not \in N$.
\end{proof}

\subsection{Weak convergence in Hilbert spaces}
\label{sec:analysisSSM}

In this section $(G,d)$ is a Hilbert space
$(\mathcal{H},\act{\cdot,\cdot})$. Under the standing assumptions the
following extended-valued function
\begin{align*}
  R(x):= \mathbb{E}\left[\norm{x-T_{\xi}x}^{2}\right] = \int_{\Omega}
  \norm{x - T_{\xi(\omega)} x}^{2} \mathbb{P}(\dd{\omega}) = \int_{I}
  \norm{x- T_{u} x }^{2} \mathbb{P}^{\xi}(\dd{u})
\end{align*}
is measurable from $\mathcal{H}$ to $[0,\infty]$. Following
\cite{Nedic2010} we use this function to characterize convergence of
the consistent fixed point problem under the weaker assumption that
the mappings $T_{\xi}$ are {\em averaged} (see \cref{d:averaged}).

\begin{lemma}[properties of $R$ and $C$ for quasi-nonexpansive
  mappings] \label{lemma:prop_R} In addition to the standing
  assumptions, suppose that $T_{i}\, (i\in I)$ is quasi-nonexpansive and continuous.
  Then
  \begin{enumerate}[label=(\roman*)]
  \item\label{item:prop_R:C} $C = R^{-1}(0)$;
  \item\label{item:prop_C:Rfinite} $R$ is finite everywhere;
  \item \label{item:prop_C:Rcont} $R$ is continuous;
  \item\label{lemma:prop_C} $C$ is convex and closed.
  \end{enumerate}
\end{lemma}

\begin{proof}
  \begin{enumerate}[label=(\roman*)]
  \item We have $x\in C$ $\Leftrightarrow$ $x\in \Fix T_{\xi}$ a.s.\
    $\Leftrightarrow$ $x=T_{\xi}x$ a.s.\ $\Leftrightarrow$ $R(x)=0$.
  \item Fix $x \in C$, then $x=T_{\xi}x$ a.s. Using
    quasi-nonexpansivity we get a.s., that
    \begin{align}
      \label{eq:dist_ineq}
      \norm{y-T_{\xi}y} \le &\norm{y-x} + \norm{x-T_{\xi} y} \le 2
      \norm{x-y}& \qquad
      \forall y\in \mathcal{H}, \\
      &\Longleftrightarrow& \nonumber \\
      \norm{y-T_{\xi}y}^{2} &\le 4\norm{y-x}^{2}& \qquad \forall
      y\in\mathcal{H}. \label{eq:dist_square_ineq}
    \end{align}
    From \eqref{eq:dist_square_ineq} it follows that $R(y)\le
    4\norm{y-x}^{2} < \infty$ for all $y \in \mathcal{H}$.
  \item Let $x, x_{n} \in \mathcal{H}$, $n \in \mathbb{N}$, with
    $x_{n} \to x$ as $n \to \infty$. Define the functions
    $f_{n}(\omega) = \norm{x_{n}-T_{\xi(\omega)}x_{n}}^{2}$ on
    $\Omega$ ($n\in\mathbb{N}$). Then, by continuity of
    $T_{\xi(\omega)}$ for fixed $\omega \in \Omega$, one has $f_{n}
    \to f := \norm{x-T_{\xi}x}^{2}$ for all $\omega \in
    \Omega$. Define the constant function $g(\omega) = 8 \epsilon^{2}
    + 8\norm{x -c}^{2}$ for some $c \in C$ and some $\epsilon >0$.  By
    \eqref{eq:dist_ineq} we have that $\norm{y-T_{\xi}y} \le 2
    \norm{y-c}$ for all $y\in \mathcal{H}$.  For $y \in
    \mathbb{B}(x,\epsilon)$ this yields $\norm{y-T_{\xi}y} \le 2
    \epsilon +2 \norm{x-c}$.  We conclude that $g$ is
    $\mathbb{P}$-integrable and $f_{n} \le g$ for all $n \in
    \mathbb{N}$ with $x_{n} \in \mathbb{B}(x,\epsilon)$.  Finally,
    application of Lebesgue's Dominated Convergence Theorem yields
    $R(x_{n}) = \mathbb{E}f_{n} \to \mathbb{E}f = R(x)$ as $n \to
    \infty$.
  \item This follows from \cite[Proposition 4.13, Proposition
    4.14]{Bauschke2011}. Note that for any $\alpha \in \mathbb{R}$,
    $a,b \in\mathcal{H}$ we have \cite[Corollary 2.14]{Bauschke2011}
    \begin{align*}
      \norm{\alpha a+(1-\alpha) b}^{2} = \alpha \norm{a}^{2} +
      (1-\alpha)\norm{b}^{2} - \alpha(1-\alpha) \norm{a-b}^{2}.
    \end{align*}
    Let $z = \lambda x+(1-\lambda)y$ with $x,y \in R^{-1}(0) = C$,
    $\lambda \in [0,1]$. One has with $T_{\xi}x=x$ and $T_{\xi}y = y$
    a.s.\ that a.s.\ holds
    \begin{align*}
      \norm{T_{\xi} z -z}^{2} &= \norm{ \lambda (T_{\xi}z -x) +
        (1-\lambda) (T_{\xi}z -y)}^{2} \\ &= \lambda \norm{T_{\xi}z -
        x}^{2} + (1-\lambda) \norm{T_{\xi}z - y}^{2} - \lambda
      (1-\lambda) \norm{x-y}^{2} \\ &\le \lambda \norm{z-x}^{2} +
      (1-\lambda) \norm{z-y}^{2} -\lambda (1-\lambda) \norm{x-y}^{2}
      \\ &= \norm{\lambda (z-x) + (1-\lambda)(z-y)}^{2} \\ &= 0.
    \end{align*}
    So $R(z) = 0$, i.e.\ $z \in R^{-1}(0)$. Closedness of $R^{-1}(0)$
    follows by continuity of $R$.
  \end{enumerate}
\end{proof}

In the next theorem we need to compute conditional
expectations of nonnegative real-valued random variables, which are
non-integrable in general (for example, if the random variable $X_0$
with distribution $\mu$ does not have a finite expectation,
$\mathbb{E}[\norm{X_{0}}]=+\infty$).  But for these random variables
the classical results on integrable random variables are still
applicable (see \cref{thm:cond_expec}), also the disintegration
theorem is still valid (see \cref{thm:disinteg}).

The stage is now set to show convergence for the
corresponding Markov chain.  The next several results concern weak
convergence of sequences of random variables with respect to the
Hilbert space, namely, $x_n \wto x$ if $\act{x_n,y} \to \act{x,y}$ for
all $y \in \mathcal{H}$.
\begin{thm}[weak cluster points belong to feasible set for averaged
  mappings]\label{thm:cvg_SSM}
  Under the standing assumptions, let $T_{i}$ be $\alpha_{i}$-averaged
  with $\alpha_{i} \le \alpha <1$ for all $i \in I$. Then weak cluster
  points (in the sense of Hilbert spaces) of the sequence $(X_{k})_{k
    \in \mathbb{N}_{0}}$ of random variables in $\mathcal{H}$
  generated by \cref{algo:SSM} are a.s.\ contained in $C$.
\end{thm}
\begin{proof}
  Fix $c \in C$. Since $T_{\xi}$ is averaged we have for all $k \in
  \mathbb{N}$ that
  \begin{align}
    \label{eq:basic_av_rel}
    \norm{X_{k+1} - c}^{2} \le \norm{X_{k} - c}^{2} -
    \frac{1-\alpha}{\alpha}\norm{X_{k+1} - X_{k}}^{2}
  \end{align}
  everywhere but on a $\mathbb{P}$-nullset $N_{c}$, which may depend
  on $c$. Let $\mathcal{F}_{k} = \sigma(X_{0},\xi_{0}, \ldots,
  \xi_{k-1})$ be the $\sigma$-algebra of all iterations of the
  algorithm up to the $k$-th and apply \cref{lemma:supermartingale}.
  We get that $\sum_{k\in\mathbb{N}_{0}} R(X_{k}) < \infty$ a.s.,
  where from \cref{thm:disinteg} follows that $\cex{\norm{X_{k+1} -
      X_{k}}^{2}}{\mathcal{F}_{k}} = R(X_{k})$. Hence there is $\tilde
  N \subset\Omega$ with $\mathbb{P}(\tilde N) = 0$ and
  $R(X_{k}(\omega)) \to 0$ as $k \to \infty$ for $\omega \in \Omega
  \setminus (N_{c}
  \cup \tilde N)$.\\
  By nonexpansiveness of $T_{\xi}$ for all we find for any $x,x_{n}
  \in \mathcal{H}$
  \begin{align*}
    \norm{x-T_{\xi}x}^{2} &= \norm{x_{n} - T_{\xi}x}^{2} +
    \norm{x-x_{n}}^{2} + 2 \act{x_{n} - T_{\xi}x,x-x_{n}} \\ &=
    \norm{x_{n} - T_{\xi}x}^{2} - \norm{x-x_{n}}^{2} + 2 \act{x -
      T_{\xi}x,x-x_{n}} \\ &= \norm{x_{n} - T_{\xi}x_{n}}^{2} +
    \norm{T_{\xi}x - T_{\xi} x_{n}}^{2} + 2\act{x_{n} - T_{\xi}
      x_{n},T_{\xi}x_{n}-T_{\xi}x} \\ &\qquad - \norm{x-x_{n}}^{2} + 2
    \act{x - T_{\xi}x,x-x_{n}} \\ &\le \norm{x_{n} - T_{\xi}x_{n}}^{2}
    + 2\act{x_{n} - T_{\xi} x_{n},T_{\xi}x_{n}-T_{\xi}x} + 2 \act{x -
      T_{\xi}x,x-x_{n}} \\ &\le \norm{x_{n} - T_{\xi}x_{n}}^{2} +
    2\norm{x_{n} - T_{\xi} x_{n}}\norm{x_{n}-x} + 2 \act{x -
      T_{\xi}x,x-x_{n}}.
  \end{align*}
  Taking expectation and using Jensen's inequality yields
  \begin{align}
    \label{eq:weak_continuity_R}
    R(x) \le R(x_{n}) + 2 \sqrt{R(x_{n})}\norm{x_{n}-x} +
    2\mathbb{E}[\act{x- T_{\xi}x,x-x_{n}}].
  \end{align}
  Now assume that the sequence $(x_{n})$ is weakly converging to $x
  \in \mathcal{H}$, i.e.\ $x_{n} \wto x$. Then the functions $f_{n} =
  \act{x-T_{\xi}x,x-x_{n}}$, $n \in \mathbb{N}$, on $\Omega$ satisfy
  $f_{n} \to 0$ a.s. Defining the $\mathbb{P}$-integrable function
  $g(\omega) \equiv \norm{x-T_{\xi(\omega)}x}\sup_{n} \norm{x-x_{n}}$
  gives us $\abs{f_{n}} \le g$ for all $n \in \mathbb{N}$ and hence by
  Lebesgue's Dominated Convergence Theorem
  $\mathbb{E}[\act{x-T_{\xi}x,x-x_{n}}] \to 0$ as
  $n \to \infty$.\\
  So for $\omega \in \Omega \setminus (N_{c} \cup \tilde N)$ there is
  a weakly convergent subsequence of the bounded sequence
  $(X_{k}(\omega))_{k \in \mathbb{N}}$, denoted $x_{n}
  :=X_{k_{n}}(\omega) \wto x_{\omega} =:x$ as $n \to \infty$.  As
  shown above this subsequence satisfies $R(x_{n}) \to 0$ as $n \to
  \infty$. We conclude with \eqref{eq:weak_continuity_R} that $R(x) =
  0$, i.e.\ $x \in C$ and hence any weak cluster point of the sequence
  $(X_{k}(\omega))_{k}$ is contained in $C$.
\end{proof}
In the case of separable Hilbert spaces, we are able to show Fej\'er
monotonicity of the sequence $(X_{k})$ a.s., so the classical theory
of convergence analysis from \cite{Bauschke2011} can be applied in
this case. An analogous statement for nonseparable Hilbert spaces
remains open since we do not have the representation
\cref{lemma:equiv_SCFP_CFP} at hand.

\begin{thm}[almost sure weak convergence under
  separability] \label{thm:weakCvg_sep} Under the same assumptions as
  in \cref{thm:cvg_SSM} assume additionally that $\mathcal{H}$ is a
  separable Hilbert space.  Then the sequence $(X_{k})$ is a.s.\
  weakly convergent (in the sense of Hilbert spaces) to a random
  variable $X_{\mu} \in C$, depending on the initial distribution
  $\mu$. Furthermore $P_{C}X_{k} \to X_{\mu}$ strongly a.s. as $k \to
  \infty$.
\end{thm}
\begin{proof}
  Instead of a nullset $N_{c}$, which may depend on $c \in C$, as in
  the proof of \cref{thm:cvg_SSM}, separability gives with help of
  \cref{lemma:equiv_SCFP_CFP} that there is a nullset $N$, such that
  on $\Omega\setminus N$ \cref{eq:basic_av_rel} is satisfied for all
  $c \in C$. This implies a.s.\ Fejér monotonicity of $(X_{k})$. Since
  from \cref{thm:cvg_SSM} follows that weak clusterpoints of $(X_{k})$
  are contained in $C$ a.s., we can now apply Theory in \cite{Bauschke2011} developed for Fej\'er
  monotone sequences, we get: From \cite[Theorem 5.5]{Bauschke2011} (a
  Fejér monotone sequence w.r.t.\ $C$ that has all weak clusterpoints
  in $C$ is weakly convergent to a point in $C$) follows that $X_{k}
  \wto X_{\mu} \in C$ a.s.

  For strong convergence of $(P_{C}X_{k})$ a.s.\ we apply
  \cite[Proposition 5.7]{Bauschke2011}. From \cite[Corollary
  5.8]{Bauschke2011} we get from $X_{k} \wto X_{\mu}$ a.s., that
  $P_{C}X_{k} \to X_{\mu}$ a.s.\ strongly as $k \to \infty$.
\end{proof}

\begin{example}[convergence to projection for affine
  subspaces] \label{example:affine_subspace} Let $\mathcal{H}$ be
  separable and $C_{i}$ be an affine subspace, $i \in I$, where $I$ is
  an arbitrary index set. Let $T_{i} = P_{i}$ be the projector onto
  $C_{i}$. Under the standing assumptions holds that $\lim_{k} X_{k} =
  X_{\mu}=P_{C}X_{0}$ for $X_{0} \sim \mu$ and any $\mu \in
  \mathscr{P}(\mathcal{H})$.

  We show, that $P_{C} X_{k+1} = P_{C} X_{k}$ for any $k \in
  \mathbb{N}_{0}$. This allows us to conclude that $P_{C}X_{k} = P_{C}X_{0}$
  for any $k \in \mathbb{N}_{0}$, and thus $P_{C} X_{0}$ is the only
  possible weak cluster point of $(X_{k})$ by \cref{thm:weakCvg_sep}.
  Using the characterization \cite[Theorem 4.1]{Deutsch2001} (if $K
  \subset \mathcal{H}$ is nonempty, closed and convex and $u \in K$
  then $\act{ x - u,k - u } \le 0$ for all $k \in K$ iff $u= P_{K}x$)
  of a projection, we find with help of \cite[Theorem
  4.9]{Deutsch2001} (for a subspace $S$ holds that
  $\act{x-P_{S}x,s}=0$ for all $s \in S$), that for $c \in C$ holds
  that
  \begin{align*}
    \act{X_{k+1} - P_{C}X_{k},c-P_{C}X_{k}} =
    \underbrace{\act{P_{\xi_{k}}X_{k} -X_{k},c-P_{C}X_{k}}}_{= 0} +
    \underbrace{\act{X_{k}-P_{C}X_{k},c-P_{C}X_{k}}}_{\le 0} \le 0.
  \end{align*}
  Hence by \cite[Theorem 4.1]{Deutsch2001} we have that $P_{C}X_{k+1}
  = P_{C} X_{k}$.
\end{example}

\subsection{Linear rates of convergence}
\label{sec:linear_rates}

We will assume in this section that $\mathcal{H}$ is a separable
Hilbert space and $T_{i}$ is $\alpha_{i}$-averaged, $i \in I$. We will
furthermore assume, that $\alpha_{i} \le \alpha$ for some
$\alpha<1$.  As with the deterministic case, geometric convergence of the algorithm
can be analyzed by introducing a condition on the set of fixed points.  In the 
context of set feasibility with finitely many sets, the condition 
is equivalent to {\em linear regularity} of the sets \cite[Assumption
2]{Nedic2011}: There exists $\kappa > 0$ such that
\begin{align}
  \label{eq:LinRegStoch}
  \dist^{2}(x,C) \le \kappa R(x) \qquad \forall x \in \mathcal{H}.
\end{align}
In the more general context of fixed point mappings, this property is
more appropriately called {\em global  metric subregularity} of $R$ at all
points in $C$ for $0$ \cite{Kruger2016};  in particular there exists a 
$\kappa>0$ such that 
\[
  \dist^{2}(x,R^{-1}(0)) \le \kappa R(x) \qquad \forall x \in \mathcal{H}.   
\]
Here $C=R^{-1}(0)$, so the above is just another way of writing \eqref{eq:LinRegStoch}.
The smallest constant
satisfying this inequality will be called the regularity constant, it
is given by
\begin{align*}
  \sup_{x \in \mathcal{H}\setminus C} \frac{\dist^{2}(x,C)}{R(x)}.
\end{align*}
\begin{thm}\label{thm:Equiv_LinRegStoch_LinCvg} In addition to the standing assumptions, suppose 
the regularity condition in \cref{eq:LinRegStoch} is
  satisfied and $T_{i}$ is $\alpha_{i}$-averaged, $i \in I$ with
  $\alpha_{i} \le \alpha$ for some $\alpha<1$.  Then the RFI
  converges geometrically in expectation to the fixed point set, i.e.\ for any
  initial distribution
  \begin{align}
    \label{eq:LinCvgInExpectation}
    \mathbb{E}[\dist(X_{k+1},C)] \le \sqrt{1 - \kappa^{-1}
      \frac{1-\alpha}{\alpha} } \mathbb{E} [\dist(X_{k},C)] \qquad
    \forall k \in \mathbb{N}_{0}.
  \end{align}
\end{thm}
\begin{proof}
  Revisiting \eqref{eq:basic_av_rel} in the proof of
  \cref{thm:cvg_SSM} gives us for $\omega \in \Omega\setminus N$ ($N$
  given by \eqref{eq:N}) and $x=P_{C}X_{k}(\omega)$
  \begin{align*}
    \dist^{2}(X_{k+1}(\omega),C) \le \norm{X_{k+1}(\omega) - x}^{2}
    \le \dist^{2}(X_{k}(\omega),C) - \frac{1-\alpha}{\alpha}
    \norm{X_{k+1}(\omega) - X_{k}(\omega)}^{2}.
  \end{align*}
  With help of Jensen's inequality and concavity of $x\mapsto
  \sqrt{x}$ on $[0,\infty)$, we get that
  \begin{align*}
    \cex{\dist(X_{k+1},C)}{\mathcal{F}_{k}} &\le \cex{\sqrt{
        \dist^{2}(X_{k},C) - \frac{1-\alpha}{\alpha}
        \norm{T_{\xi_{k}}X_{k} - X_{k}}^{2}}}{\mathcal{F}_{k}} \\
    &\le \sqrt{ \dist^{2}(X_{k},C) - \frac{1-\alpha}{\alpha}
      \cex{\norm{T_{\xi_{k}}X_{k} - X_{k}}^{2}}{\mathcal{F}_{k}}} \\
    &= \sqrt{ \dist^{2}(X_{k},C) - \frac{1-\alpha}{\alpha} R(X_{k})} \\
    &\le \sqrt{1- \kappa^{-1} \frac{1-\alpha}{\alpha}} \dist(X_{k},C).
  \end{align*}
\end{proof}
\noindent Note that it could be $\mathbb{E}[\dist(X_{k},C)]=\infty$
for all $k \in \mathbb{N}$, depending on the initial distribution
$\mu$.

The next theorem concerns the {\em Wasserstein distance} of two
probability measures.  For two measures $\nu_{1},\nu_{2} \in
\mathscr{P}(G)$ this is given by
\begin{align*}
  W(\nu_{1},\nu_{2}) = \inf_{\substack{ Y_{1} \sim \nu_{1} \\ Y_{2}
      \sim \nu_{2}}} \mathbb{E}[\norm{Y_{1} - Y_{2}}].
\end{align*}

\begin{thm}[strong convergence and geometric convergence of
  measures] \label{thm:LinCVGMeasure} Under the standing assumptions,
  suppose the regularity condition in \cref{eq:LinRegStoch} is
  satisfied and $T_{i}$ is $\alpha_{i}$-averaged, $i \in I$ with
  $\alpha_{i} \le \alpha$ for some $\alpha<1$. Then $X_{k} \to X$
  strongly a.s.\ as $k \to \infty$ and the Wasserstein distances
  $W(\mathcal{L}(X_{k}),\mathcal{L}(X))$ also converge geometricly,
  there is $r \in (0,1)$ such that
  \begin{align*}
    W(\mathcal{L}(X_{k}),\mathcal{L}(X)) \le 2 r^{k}
    W(\mathcal{L}(X_{0}),\mathcal{L}(X)).
  \end{align*}
\end{thm}
\begin{proof}
  See also \cite[Theorem 5.12]{Bauschke2011}. One has a.s.\ that
  \begin{align*}
    \norm{X_{k} - X_{k+m}} \le \norm{X_{k} - P_{C}X_{k}} +
    \norm{P_{C}X_{k} - X_{k+m}} \le 2 \dist(X_{k},C) \le 2
    \sqrt{\kappa R(X_{k})}.
  \end{align*}
  We used here, that $T_{\xi}$ is nonexpansive and it satisfies
  $T_{\xi} c = c$ for any $c \in C$ a.s., hence $\norm{P_{C}X_{k} -
    X_{k+m}} = \norm{T_{\xi_{k+m-1}}\cdots T_{\xi_{k}} P_{C}X_{k} -
    X_{k+m}} \le \dist(X_{k},C)$.  This gives us that $(X_{k})$ is a
  Cauchy sequence a.s., since $R(X_{k}) \to 0$ as seen in the proof of
  \cref{thm:cvg_SSM}. It's limit $X$ is contained in $C$, since it's
  weak limit needs to coincide with the strong limit. Letting $m \to
  \infty$ one arrives at $ \norm{X_{k} - X} \le 2
  \dist(X_{k},C)$. Taking the expectation yields
  $\mathbb{E}[\norm{X_{k} - X}] \le 2\mathbb{E}[\dist(X_{k},C)]$.
  Hence, using \cref{thm:Equiv_LinRegStoch_LinCvg} gives us
  $\mathbb{E}[\norm{X_{k} - X}] \le 2 r^{k}
  \mathbb{E}[\dist(X_{0},C)]$ with $r = \sqrt{1 - \kappa^{-1}
    \tfrac{1-\alpha}{\alpha} }$ and using the fact that
  $\mathbb{E}[\dist(X_{0},C)] \le
  W(\mathcal{L}(X_{0}),\mathcal{L}(X))$, we have, by the definition of
  the Wasserstein distance,
  \begin{align*}
    W(\mathcal{L}(X_{k}),\mathcal{L}(X)) \le 2 r^{k}
    W(\mathcal{L}(X_{0}),\mathcal{L}(X)).
  \end{align*}
\end{proof}
\noindent Note that it could be $W(\mathcal{L}(X_{0}),\mathcal{L}(X))
= \infty$, depending on the initial distribution $\mu$.

\begin{rem}[$\epsilon$-fixed point]
  In order to assure that, with probability greater than $1-\beta$,
  the $k$-th iterate is in an $\epsilon$ neighborhood of the feasible
  set $C$, it is sufficient that $k \ge \ln(\tfrac{\beta
    \epsilon}{\sqrt{\kappa R(x)}}) / \ln(c)$, where
  $c=\sqrt{1-\frac{1-\alpha}{\alpha}\kappa^{-1}}$ and $X_{0} \sim
  \delta_{x}$.  To see this, note that, by Markov's inequality,
  \begin{align*}
    \mathbb{P}(X_{k} \in C+\epsilon \mathbb{B}(0,1)) &=
    \mathbb{P}(\dist(X_{k},C) < \epsilon) \\ &=
    1-\mathbb{P}(\dist(X_{k},C) \ge \epsilon) \\ &\ge 1-
    \frac{\mathbb{E}[\dist(X_{k} ,C)]}{\epsilon} \\ &\ge
    1-r^{k}\frac{\dist(x,C)}{\epsilon} \\ &\ge 1- r^{k}
    \frac{\sqrt{\kappa R(x)}}{\epsilon}.
  \end{align*}
\end{rem}
\begin{rem}
  As seen in \cref{example:affine_subspace} the probability
  $\mathbb{P}(X_{k} \in C)$ can increase to $1$ as $k \to \infty$, but
  this is not necessarily the case, as we will see in Examples
  \ref{example:intervals} and \ref{example:disks}.  There, one finds
  that $\mathbb{P}(X_{k} \in C) = \mathbb{P}(X_{0} \in C)$ for
  $k\in\mathbb{N}$. In \cref{example:lines} it holds that
  $\mathbb{P}(X_{k} \in C) = \mathbb{P}(X_{1} \in C)$ for all $k\in
  \mathbb{N}$.
\end{rem}

\begin{thm}[necessary and sufficient conditions for geometric convergence]\label{thm:Equiv_lin_cvg}
  Under the standing assumptions, let $T_{i}$ be $\alpha_{i}$-averaged, $i
  \in I$ with $\alpha_{i} \le \alpha$ for some $\alpha<1$. The
  regularity condition in \cref{eq:LinRegStoch} is satisfied if and
  only if there exists $r \in [0,1)$ such that
  \begin{align}
    \label{eq:eq_condition_LinReg}
    \mathbb{E} [\dist(T_{\xi}x,C)] \le r \dist(x,C)\qquad \forall x
    \in \mathcal{H}.
  \end{align}
  Furthermore, condition \cref{eq:LinRegStoch} is necessary and
  sufficient for geometric convergence in expectation of \cref{algo:SSM}
  to the fixed point set $C$ as in \cref{eq:LinCvgInExpectation} with
  a uniform constant for all initial probability measures.
\end{thm}
\begin{proof}
  \cref{eq:LinRegStoch} implies \cref{eq:LinCvgInExpectation}, which
  in turn implies \cref{eq:eq_condition_LinReg} (with $X_{0} \sim
  \delta_{x}$) by \cref{thm:Equiv_LinRegStoch_LinCvg} with
  $r=\sqrt{1-\kappa^{-1} \tfrac{1- \alpha}{\alpha}}$. The other
  implication follows the same proof pattern as \cite[Theorem
  3.11]{Luke2016b}. We note that, by \cref{thm:disinteg}, if $X_{0}
  \sim \delta_{x}$ for $x \in \mathcal{H}$, then
  \begin{align*}
    \cex{\norm{X_{1} - X_{0}}}{\xi_{0}} = \norm{T_{\xi_{0}}x-x},
  \end{align*}
  hence by H\"older's inequality
  \begin{align*}
    \mathbb{E}[\norm{X_{1}-X_{0}}] \le \sqrt{R(x)}.
  \end{align*}
  Furthermore we can estimate
  \begin{align*}
    \norm{X_{1} - X_{0}} = \norm{X_{1} - P_{C}X_{1}+P_{C}X_{1} -
      X_{0}} \ge \dist(X_{0},C) - \dist(X_{1},C).
  \end{align*}
  Taking the expectation above, the assumption that $\mathbb{E}
  [\dist(X_{1},C)] \le r \mathbb{E}[\dist(X_{0},C)]$ yields
  \begin{align*}
    (\forall x \in \mathcal{H}) \qquad R(x) \ge (1-r)^{2}
    \dist^{2}(x,C),
  \end{align*}
  i.e. the constant $\kappa$ in \cref{eq:LinRegStoch} is finite with \
  $\kappa \le (1-r)^{-2} < \infty$. So \cref{eq:eq_condition_LinReg}
  implies \cref{eq:LinRegStoch}.
 
  For the last implication of the theorem, note that, in case
  \cref{eq:LinCvgInExpectation} is satisfied with the same constant $r
  \in (0,1)$ for all Dirac measures $\delta_{x}$ with $x \in
  \mathcal{H}$, then \cref{eq:eq_condition_LinReg} also holds (letting
  $X_{0} \sim \delta_{x}$) and hence by the above equivalence
  \cref{eq:LinRegStoch} is satisfied.  This completes the proof.
\end{proof}

\begin{rem}
  Conventional analytical strategies invoke strong convexity in order
  to achieve geometric convergence.  Our analysis makes no such
  assumption on the sets $C_i$.  \cref{thm:Equiv_lin_cvg} shows that
  geometric convergence is a by-product, mainly, of the regularity of the
  set of fixed points.  The results of \cite{Luke2016b} indicate that
  one could formulate a necessary regularity condition for {\em
    sublinear} convergence, which also might be useful for stochastic
  algorithms.
\end{rem}

\section{Applications}

We specialize the framework above to several well-known settings:
consistent convex feasibility, linear operator equations and in
particular Hilbert-Schmidt operators (i.e. linear integral equations).

\subsection{Feasibility and stochastic projections}
\label{sec:modesofcvg}

There are many algorithms for solving convex feasibility problems.  We
focus on the (conceptually) simplest of these, namely stochastic
projections.  In the context of \cref{algo:SSM}, $T_{i} = P_{i}$ is a
projector, $i \in I$, onto a nonempty closed and convex set $C_{i}
\subset \mathcal{H}$, $i \in I$ and $\mathcal{H}$ a Hilbert
space. Note that projectors are $\tfrac{1}{2}$-averaged operators
\cite[Proposition 4.8]{Bauschke2011} (also referred to as \emph{firmly
  nonexpansive} operators), so $\alpha_{i}=\tfrac{1}{2}$ for all $ i
\in I$, we then can choose the upper bound $\alpha=\tfrac{1}{2}$ as
well. Also note that $\Fix P_{i} = C_{i}$, $i \in I$.

As a first assertion we give an equivalent characterization
for the regularity property in \cref{eq:LinRegStoch} using just
properties of $R$. This characterization, known as
Kurdyka-\L{}ojasiewicz (KL) property, eliminates the term with the
distance to the usually unknown fixed point set $C$, but one needs to
be able to compute the first derivative of the function $R$.  For
convex sets this is unproblematic since $R$ is the expectation of the
squared distances to the convex sets $C_{i}$, see
\cref{lemma:further_prop_R}.

\begin{definition}[KL property]
  A convex, continuously differentiable function $\mymap{f}{\mathcal{H}}{\mathbb{R}}$ with $\inf_{x}
  f(x) =0$ and $S:= \argmin f \neq \emptyset$ is said to have the global KL
  property, if there exists a concave continuously differentiable
  function $\mymap{\varphi}{\mathbb{R}_{+}}{\mathbb{R}_{+}}$ with
  $\varphi(0)=0$ and $\varphi' >0$ such that
  \begin{align*}
    \varphi'(f(x)) \norm{\nabla f(x)} \ge 1 \qquad \forall x \in
    \mathcal{H} \setminus S.
  \end{align*}
\end{definition}

The following theorem is a direct consequence of \cite{Bolte2017}.
\begin{prop}[equivalent characterization of
  \cref{eq:LinRegStoch}]\label{thm:equiv_linConv_propR}
  Under the standing assumptions, let $T_{i}=P_{i}$ be projectors onto
  nonempty, closed and convex sets, $i \in I$. Then the regularity
  condition in \cref{eq:LinRegStoch} is satisfied with $\kappa>0$ if
  and only if $R(x) \le \frac{\kappa}{4} \norm{\nabla R(x)}^{2}$
  $\forall x \in \mathcal{H}$, i.e.\ $R$ has the global KL property.
\end{prop}
\begin{proof}
  Apply \cite[Corollary 6]{Bolte2017} with $\varphi(s):=
  \sqrt{\kappa s}$ and $f=R$ and note that $R$ is convex and
  differentiable (see \cref{lemma:further_prop_R}).
\end{proof}

\begin{thm}[uniform bounds]\label{thm:PropCond}
  Under the standing assumptions, suppose the regularity condition in
  \cref{eq:LinRegStoch} is satisfied and that $\mathcal{H}$ is
  separable and $T_{i}=P_{i}$ are projectors onto nonempty, closed and
  convex sets, $i \in I$. Then the probability of any point being
  feasible is uniformly bounded, i.e.\ $\mathbb{P}(x \in C_{\xi})\le r
  <1$ for all $x \in \mathcal{H}\setminus C$.
\end{thm}
\begin{proof}
  It holds surely for all $x \in \mathcal{H}$
  \begin{align*}
    \dist(P_{\xi} x,C) \ge \dist(x,C) - \dist(x,C_{\xi}).
  \end{align*}
  This, together with the expectation
  \begin{align*}
    \mathbb{E}[\dist(x,C_{\xi})]= \mathbb{E}[\dist(x,C_{\xi})
    \1_{\{x\notin C_{\xi}\}}] \le \mathbb{E}[\dist(x,C) \1_{\{x\notin
      C_{\xi}\}}]= \dist(x,C) (1-\mathbb{P}(x\in C_{\xi}))
  \end{align*}
  yields, for $X_{0} \sim \delta_{x}$,
  \begin{align*}
    \mathbb{E}[\dist(X_{1},C)] \ge \mathbb{P}(x \in C_{\xi})
    \dist(x,C).
  \end{align*}
  Hence by \cref{thm:Equiv_lin_cvg}
  \begin{align*}
    1>r:=\sup_{x \in\mathcal{H}\setminus C}
    \frac{\mathbb{E}[\dist(X_{1},C)]}{\dist(x,C)} \ge \sup_{x \in
      \mathcal{H} \setminus C} \mathbb{P}(x \in C_{\xi}).
  \end{align*}
\end{proof}

\begin{thm}[finite vs. infinite convergence]\label{thm:no_finite_cvg}
  Under the standing assumptions, let $\mathcal{H}$ be separable and
  let $T_{i}=P_{i}$ be projectors ($i \in I$).  Then one of the
  following holds:
  \begin{enumerate}[label=(\roman*)]
  \item $\mathbb{P}(X_{1} \in C) = 1$ and $\mathbb{P}(X_{n} \in C) =
    1$ for all $n \in \mathbb{N}$,
  \item $\mathbb{P}(X_{1} \in C) < 1$ and $\mathbb{P}(X_{n}\in C)<1$
    for all $n \in \mathbb{N}$.
  \end{enumerate}
\end{thm}
\begin{proof}
  \begin{enumerate}[label=(\roman*)]
  \item If $\mathbb{P}(X_{1} \in C) = 1$, then $X_{k} = X_{1}$ a.s.\
    for all $k \ge 1$.
  \item From $\int p(x,C) \mu(\dd{x})=\mathbb{P}(X_{1} \in C) < 1$ we
    get, that there is $x \in \supp \mu \setminus C$ with $p(x,C)<1$,
    where $\mu$ is the initial distribution. Since $p(x,\mathcal{H}
    \setminus C)>0$, there exists $y \in \supp p(x,\cdot) \setminus
    C$.  Then by \cref{thm:supp_measure} this implies that
    $p(x,\cb(y,\epsilon))>0$ for all $\epsilon >0$.\\
    Furthermore, one has for any $\epsilon>0$ that
    \begin{align}
      \label{eq:no_finite_cvg_prob_estimate}
      (\forall z \in \cb(y, \epsilon))\qquad p(z,\cb(y,2\epsilon)) \ge
      p(x,\cb(y,\epsilon))>0.
    \end{align}
    To see this, note that, for $\omega \in M(\epsilon):=\mysetc{
      \omega \in \Omega}{P_{\xi(\omega)} x \in \cb(y,\epsilon)}$, we
    have
    \begin{align*}
      \norm{P_{\xi(\omega)} z - y} &\le
      \norm{P_{\xi(\omega)}z-P_{\xi(\omega)}y}+\norm{P_{\xi(\omega)}
        y- y} \\ &\le \norm{z-y} + \norm{P_{\xi(\omega)} y -y} \\ &\le
      \norm{z-y} + \norm{P_{\xi(\omega)} x -y}\\ &\le 2 \epsilon.
    \end{align*}
    Here we have used nonexpansiveness of $P_{\xi}$ and the definition
    of a projection.  Now \eqref{eq:no_finite_cvg_prob_estimate}
    follows from the identity
    $\mathbb{P}(M_{k}(\epsilon))=p(x,\cb(y,\epsilon))>0$.\\
    Furthermore, one has for $\epsilon > 0$ that
    \begin{align}
      \label{eq:no_finite_cvg_prob_estimate2}
      (\forall w\in \cb(x,\epsilon))\qquad p(w,\cb(y,2\epsilon)) \ge
      p(x,\cb(y,\epsilon))>0.
    \end{align}
    To see this, note that for $\omega \in M(\epsilon)$, we have
    \begin{align*}
      \norm{P_{\xi(\omega)} w - y} &\le \norm{P_{\xi(\omega)} w -
        P_{\xi(\omega)} x} + \norm{P_{\xi(\omega)} x - y} \\
      &\le 2\epsilon.
    \end{align*}
    Now, fix $\epsilon>0$ such that both $\cb(y,\epsilon)\cap C =
    \emptyset$ and $\cb(x,\epsilon)\cap C=\emptyset$. We get for any
    $w \in \mathcal{H}$ and $n \in \mathbb{N}$
    \begin{align*}
      p^{n+1}(w,\cb(y,\epsilon)) \ge \int_{\cb(y,\tfrac{\epsilon}{2})}
      p(z,\cb(y,\epsilon)) p^{n}(w,\dd{z}) \ge
      p(x,\cb(y,\tfrac{\epsilon}{2}))
      p^{n}(w,\cb(y,\tfrac{\epsilon}{2})).
    \end{align*}
    So iteratively, denoting $\epsilon_{n}:=2^{-n}\epsilon$, we arrive
    at
    \begin{align*}
      p^{n+1}(w,\cb(y,\epsilon)) \ge \prod_{i=1}^{n}
      p(x,\cb(y,\epsilon_{i})) p(w,\cb(y,\epsilon_{n})).
    \end{align*}
    The last probability can be estimated for $w \in
    \cb(x,\epsilon_{n+1})$ by \eqref{eq:no_finite_cvg_prob_estimate2}
    through
    \begin{align*}
      p(w,\cb(y,\epsilon_{n})) \ge p(x,\cb(y,\epsilon_{n+1})).
    \end{align*}
    Summarizing, we have that $p^{n}(w,\cb(y,\epsilon))$ is locally
    uniformly bounded from below for $w \in \cb(x,\epsilon_{n})$. That
    implies
    \begin{align*}
      \mathbb{P}(X_{n} \in \mathcal{H}\setminus C) &=
      \int_{\mathcal{H}} p^{n}(w,\mathcal{H} \setminus C) \mu(\dd{w})
      \\ &\ge \int_{\cb(x,\epsilon_{n})} p^{n}(w,\cb(y,\epsilon))
      \mu(\dd{w}) \\ &\ge [p(x,\cb(y,\epsilon_{n}))]^{n}
      \mu(\cb(x,\epsilon_{n})) > 0,
    \end{align*}
    i.e.\ $\mathbb{P}(X_{n} \in C)<1$ for all $n\in\mathbb{N}$, as
    claimed.
  \end{enumerate}
\end{proof}

\begin{rem}
  \cref{thm:no_finite_cvg} can be interpreted as a lower bound on the
  complexity of the RFI analogous to the deterministic case
  \cite[Theorem 5.2]{Luke2016b}, where the alternating projection
  algorithm converges either after one iteration or after infinitely
  many.  Alternatively, the {\em stopping} or {\em hitting time} of a
  process is defined as
  \[
  T\equiv \inf\mysetc{n}{X_n\in C}.
  \]
  In this context, \cref{thm:no_finite_cvg} says that, either
  $\mathbb{P}(T=1)=1$ or $\mathbb{P}(T=n)<1$ for all
  $n\in\mathbb{N}$. Note, it could happen that
  $\mathbb{P}(T=\infty)=1$, in which case $\mathbb{P}(T=n)=0$ for all
  $n\in\mathbb{N}$.

\end{rem}

\begin{example}[finite and infinite
  convergence] \label{example:two_orth_halfspaces} With just two sets,
  the deterministic alternating projections algorithms can converge in
  finitely many steps.  But when the projections onto the respective
  sets are randomly selected, convergence might only come after
  infinitely many steps.  For example, let $C_{1}=\mathbb{R}_{+}\times
  \mathbb{R}$ and $C_{2}=\mathbb{R}\times \mathbb{R}_{+}$ and
  $\mathbb{P}(\xi=1)=0.3$, $\mathbb{P}(\xi=2)=0.7$. Then $C =
  \mathbb{R}_{+}\times \mathbb{R}_{+}$. Set $\mu=\delta_{x}$, where
  $x=\icol{-1 \\ -1}$. Then $\mathbb{P}(X_{1}\in C) = 0$ or more
  generally $\mathbb{P}(X_{n} \in C) = 1 - 0.3^{n}- 0.7^{n}<1$, $n \in
  \mathbb{N}$.  Now let $\mathbb{P}(\xi=1)=1$ and
  $\mathbb{P}(\xi=2)=0$, then $C=C_{1}$ and for $\mu$ as above
  $\mathbb{P}(X_{1} \in C)=1$ and so $\mathbb{P}(X_{n} \in C)=1$.
\end{example}

\begin{example}[no uniform geometric
  convergence] \label{example:intervals} In this example we show a
  sublinear convergence rate for infinitely many overlapping
  intervals.  This is in contrast to the convergence properties of
  finitely many intervals with nonempty interior, where one would
  expect a geometric rate.

  Let $\xi \sim \mathrm{unif}[\epsilon-\tfrac{1}{2},
  \tfrac{1}{2}-\epsilon]$ for some $\epsilon \in
  [0,\tfrac{1}{2})$. Define the nonempty and closed intervals $C_{r} =
  [r-\tfrac{1}{2},r+\tfrac{1}{2}]$, $r \in \mathbb{R}$.

  \begin{minipage}{1.0\linewidth}
    \centering
    \includegraphics[width=0.4\linewidth]{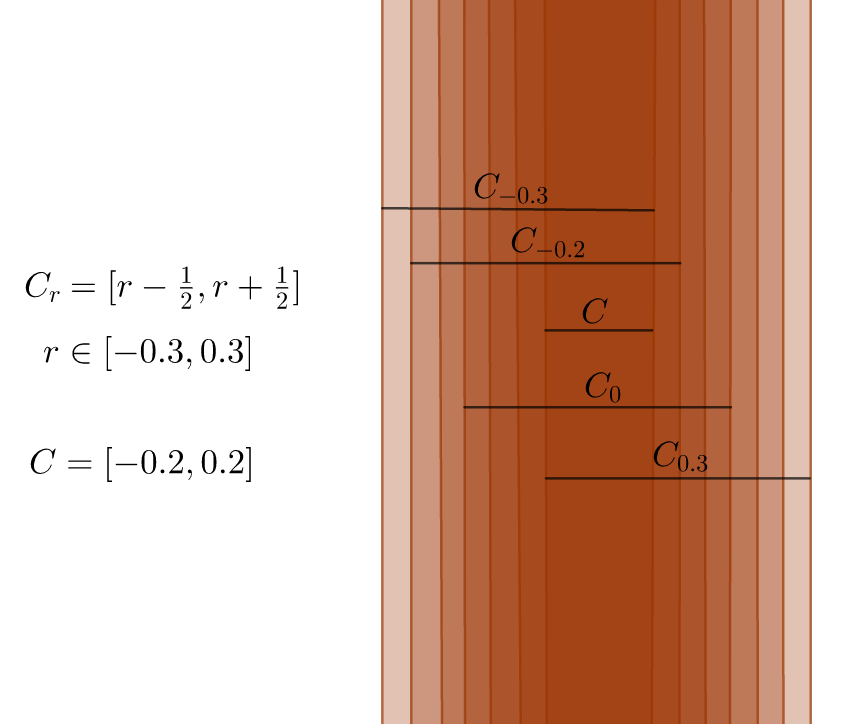}
    \captionof{figure}{}
  \end{minipage}

\noindent The projector onto these intervals is given by
\begin{align*}
  P_{r}x =
  \begin{cases}
    r+\tfrac{1}{2} & x \ge r+\tfrac{1}{2} \\
    r - \tfrac{1}{2} & x \le r - \tfrac{1}{2} \\
    x & r-\tfrac{1}{2} \le x \le r + \tfrac{1}{2} \\
  \end{cases}.
\end{align*}
The Lebesgue-density $\rho_{\epsilon}$ of $\xi$ is
\begin{align*}
  \rho_{\epsilon}(y) = \frac{1}{1-2\epsilon}
  \mathds{1}_{[\epsilon-\tfrac{1}{2}, \tfrac{1}{2}-\epsilon]}(y).
\end{align*}
One can compute
\begin{align*}
  R_{\epsilon}(x) &= \mathbb{E}[\abs{P_{\xi}x - x}^{2}] =
  \int_{\mathbb{R}}\abs{P_{r}x-x}^{2} \rho_{\epsilon}(r)\dd{r} =
  \frac{1}{1-2\epsilon} \int_{\epsilon-\tfrac{1}{2}}
  ^{\tfrac{1}{2}-\epsilon} \abs{P_{r} x-x}^{2} \dd{r} \\
  &= \frac{1}{1-2\epsilon} \mathds{1}_{[\epsilon,\infty)}(\abs{x})
  \frac{ (\abs{x} -\epsilon)^{3} +\min(1-\abs{x}-\epsilon,0)^{3} }{3}.
\end{align*}
Now, let us examine regularity properties. For the case $\epsilon\in
[0,\tfrac{1}{2})$, the problem is a consistent feasibility
problem with $C=[-\epsilon,\epsilon]$. While the regularity condition
in \cref{eq:LinRegStoch} is trivially satisfied for $\abs{x}\le
\epsilon$, for $\epsilon\le \abs{x} \le 1-\epsilon$ we find
$R_{\epsilon}(x) =
\tfrac{1}{1-2\epsilon}\tfrac{(\abs{x}-\epsilon)^{3}}{3}$ and
$\dist^{2}(x,C) = (\abs{x}-\epsilon)^{2}$. So the regularity property in
\cref{eq:LinRegStoch} is not satisfied for any $\kappa > 0$ here. That
means by \cref{thm:Equiv_lin_cvg}, that we cannot expect uniform
geometric convergence (i.e.\ there is no $r \in [0,1)$ with
$\mathbb{E}[\dist(X_{k+1},C)] \le r \mathbb{E}[\dist(X_{k},C)]$, where
$X_{0} \sim \delta_{x}$, $x \in \mathcal{H}$).
\end{example}
\begin{example}[uniform geometric convergence] \label{example:lines} We
  provide here a concrete example where geometric convergence of the
  RFI is achieved.  This is somewhat
  surprising since the angle between the sets can become arbitrarily
  small.  In the deterministic setting, this results in arbitrarily
  slow convergence of the algorithm.  This provides some intuition for
  why stochastic algorithms can outperform deterministic variants.

  Let $C_{\alpha} := \mathbb{R} e_{\alpha}$ with
  $e_{\alpha}=\icol{\cos(\alpha)\\ \sin(\alpha)}$, $\alpha \in
  [0,2\pi)$ and let $\xi \sim \mathrm{unif}[0,\beta]$, where $\beta
  \in (0,\tfrac{\pi}{2})$.

  \begin{minipage}{1.0\linewidth}
    \centering
    \includegraphics[width=0.4\linewidth]{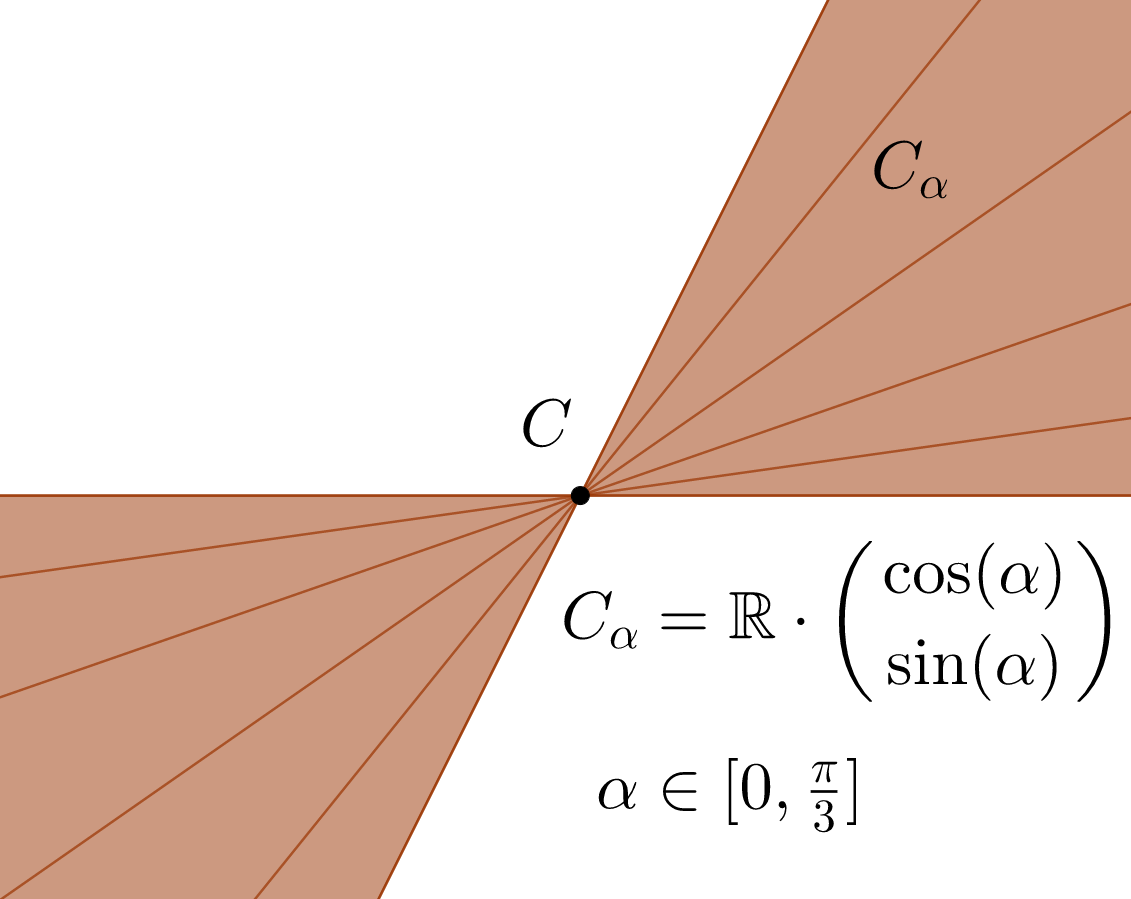}
    \captionof{figure}{}
  \end{minipage}

  \noindent We have $C = \{0\}$, so $\dist(x,C)=\norm{x}$ and the
  density of $\xi$ is $\rho_{\beta}(\alpha) =
  \tfrac{1}{\beta}\mathds{1}_{[0,\beta]} (\alpha)$. The projector onto
  the linear subspace $C_{\alpha}$ of $\mathbb{R}^{2}$ is given by
  \begin{align*}
    P_{\alpha}(x) = x - \act{
      \begin{pmatrix}
        \sin(\alpha) \\
        -\cos(\alpha)
      \end{pmatrix}, x}
    \begin{pmatrix}
      \sin(\alpha) \\
      -\cos(\alpha)
    \end{pmatrix}.
  \end{align*}
  We find then
  \begin{align*}
    R_{\beta}(x) &= \frac{1}{\beta} \int_{0}^{\beta}
    \norm{P_{\alpha}x-x}^{2} \dd{\alpha} = \frac{1}{\beta}
    \int_{0}^{\beta} (x_{1} \sin(\alpha) -x_{2}\cos(\alpha))^{2}
    \dd{\alpha} \\ &= \frac{1}{\beta}\left[ x_{1}^{2}
      \left(\frac{\beta-\sin(\beta)\cos(\beta)}{2}\right) + x_{2}^{2}
      \left(\frac{\beta+\sin(\beta)\cos(\beta)}{2}\right) - x_{1}x_{2}
      \sin^{2}(\beta)\right].
  \end{align*}
  Using that for $x=\lambda e_{\alpha}$ with $\lambda\ge 0$ holds
  $\dist(x,C) = \lambda$ and $R_{\beta}(x) =
  \lambda^{2}R_{\beta}(e_{\alpha})$ and employing trigonomertric
  calculation rules, we find the regularity constant in
  \cref{eq:LinRegStoch} not to be smaller than
  \begin{align*}
    \kappa=\sup_{x \in \mathbb{R}^{2}}
    \frac{\dist^{2}(x,C)}{R_{\beta}(x)} = \sup_{\alpha \in [0,2\pi)}
    \frac{8\beta}{2\beta -\sin(2\beta-2\alpha)-\sin(2\alpha)} =
    \frac{4\beta}{\beta-\sin(\beta)},
  \end{align*}
  where the last supremum is attained at $\alpha=\tfrac{\beta}{2}$. So
  from \cref{thm:Equiv_LinRegStoch_LinCvg} we get uniform geometric
  convergence.
\end{example}

\begin{example}[disks on a circle] \label{example:disks} This example
  illustrates \cref{thm:PropCond}.  Let $C_{\alpha}:= \cb(\rho
  e_{\alpha},1) \subset \mathbb{R}^{2}$, where $0<\rho<1$ and
  $e_{\alpha}=\icol{\cos(\alpha)\\ \sin(\alpha)}$, $\alpha \in
  [0,2\pi)$ and let $\xi \sim \mathrm{unif}[0,2\pi]$.

  \begin{minipage}{0.49\linewidth}
    \centering
    \includegraphics[width=\linewidth]{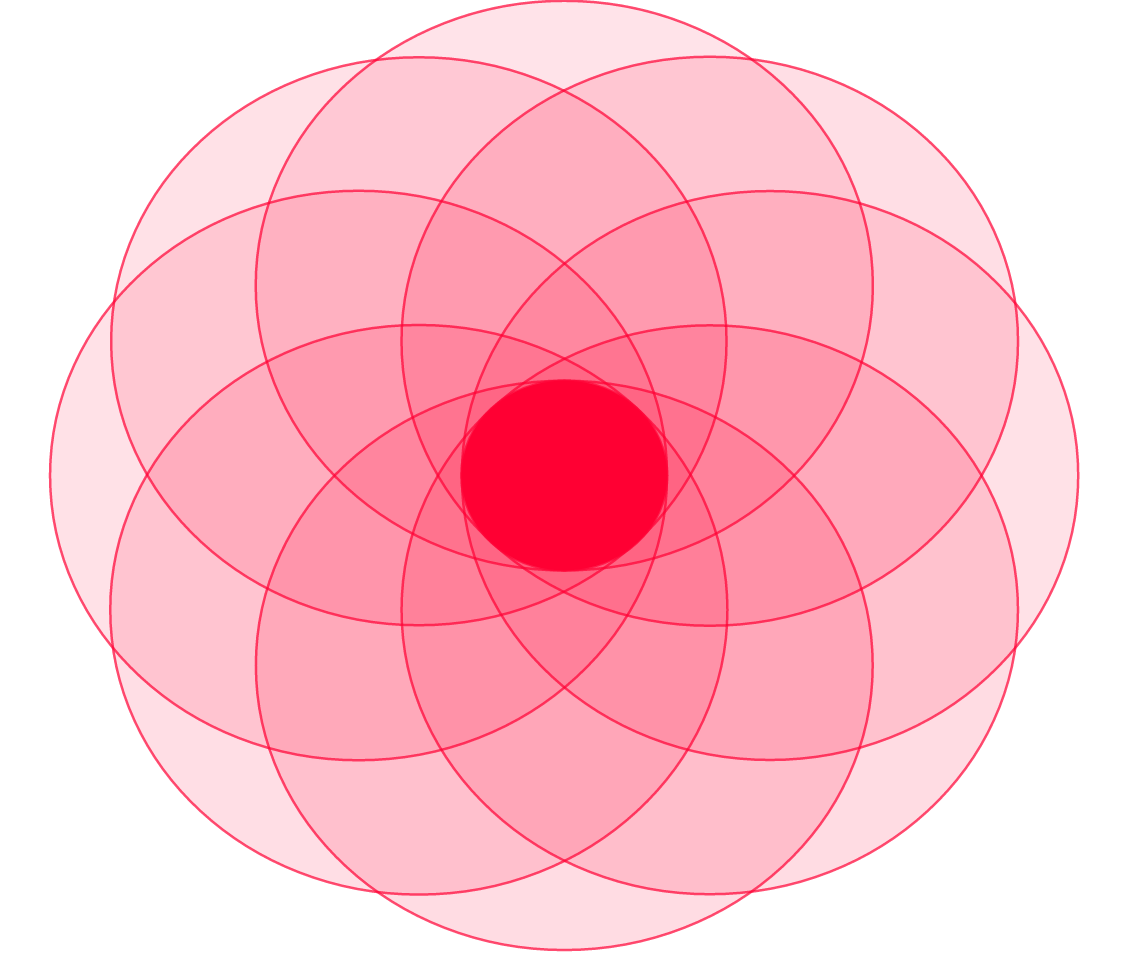}
    \captionof{figure}{}
  \end{minipage}
  \begin{minipage}{0.49\linewidth}
    \centering
    \includegraphics[width=\linewidth]{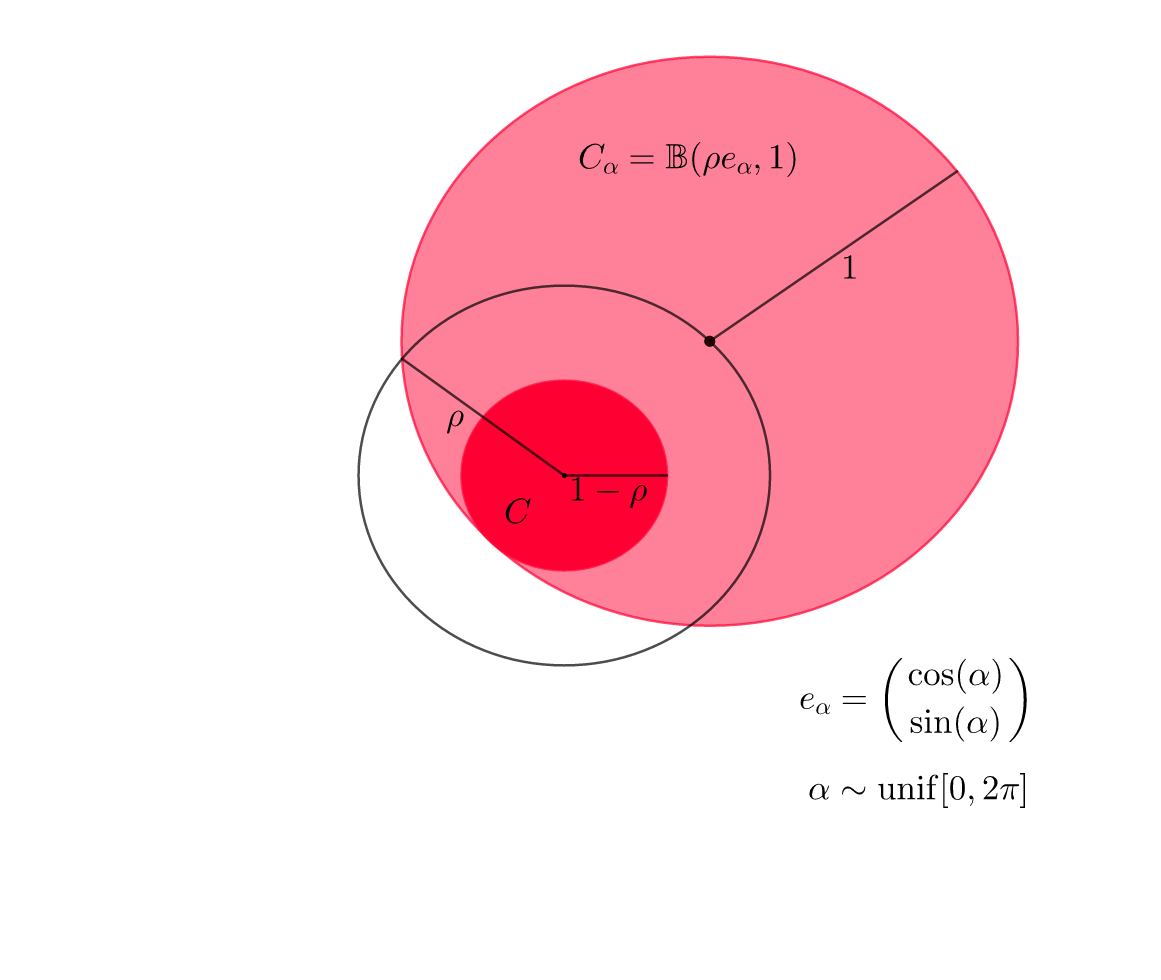}
    \captionof{figure}{}
  \end{minipage}

  \noindent The intersection is given by $C = \cb(0,1-\rho)$. We show
  next that sets with this configuration do not satisfy
  \eqref{eq:LinRegStoch}.  To see this we show that there is a
  sequence $(x_{n})_{n} \subset \mathbb{R}^{2}$ with $\mathbb{P}(x_{n}
  \in C_{\xi}) \to 1$ as $n \to \infty$.  By \cref{thm:PropCond} we
  conclude that \eqref{eq:LinRegStoch} cannot hold.  Indeed, let $x =
  x(\lambda)=\lambda \icol{1\\ 0}$ with $ \lambda \ge 1-\rho$, then
  \begin{align*}
    \mathbb{P}(x\in C_{\xi}) &= \frac{1}{2\pi} \int_{0}^{2\pi}
    \mathds{1}\{\norm{x-\rho e_{\alpha}} \le 1\} \dd{\alpha} \\ &=
    \frac{1}{2\pi} \int_{0}^{2\pi}
    \mathds{1}\{\lambda^{2}+\rho^{2}-2\lambda\rho \cos(\alpha) \le 1\}
    \dd{\alpha} \\ &= \frac{1}{2\pi} \int_{-\beta}^{\beta} 1
    \dd{\alpha} \\ &= \frac{\beta}{\pi},
  \end{align*}
  where $\beta = \beta(\lambda) = \cos^{-1}\left(
    \tfrac{\lambda^{2}+\rho^{2}-1}{2\lambda\rho} \right)$, if
  $\lambda\le 1+\rho$. We have $\beta(\lambda) \to \pi$ as $\lambda
  \to 1-\rho$, so $\mathbb{P}(x(\lambda) \in C_{\xi}) \to 1$ as
  $\lambda\to 1-\rho$.

  In contrast to the case where $\rho\in(0,1)$, the extreme cases
  where $\rho=0$ and $\rho=1$ {\em do } satisfy
  \eqref{eq:LinRegStoch}. This will not be shown
  here.
\end{example}

\subsection{RFI with two families of mappings}
\label{sec:twoSets}
The set feasibility examples above lead very naturally to the more
general context of mappings $\mymap{T_{i}}{G}{G}$, $i \in I$ and
$\mymap{S_{j}}{G}{G}$, $j \in J$ on a metric space $(G,d)$, where
$I,J$ are arbitrary index sets. Here we envision the scenario where
$C_T := \mysetc{x \in G}{\mathbb{P}(x \in \Fix T_{\xi}) = 1}$ and $C_S
:= \mysetc{x \in G}{\mathbb{P}(x \in \Fix S_{\zeta}) = 1}$ are
distinctly different sets, possibly nonintersecting.  Let
$(\Omega,\mathcal{F},\mathbb{P})$ be a probability space and let
$\mymap{\xi}{\Omega}{I}$, $\mymap{\zeta}{\Omega}{J}$ be two random
variables. Let $(\xi_{n})_{n\in\mathbb{N}}$ be an iid sequence with
$\xi_{n} \overset{\text{d}}{=} \xi$ and $(\zeta_{n})_{n \in
  \mathbb{N}}$ iid with $\zeta_{n} \overset{\text{d}}{=} \zeta$. The
two sequences are assumed to be independent of each other.  Let $\mu$
be a probability measure on $(G, \mathcal{B}(G))$. Consider the
stochastic selection method for two families of mappings
\begin{algorithm}[H]
  \caption{RFI for two families of mappings}\label{algo:SSM2}
  \begin{algorithmic}[0]
    \Require{$X_{0} \sim \mu$} \For{$k=0,1,2,\ldots$} \State{ $
      X_{k+1} = S_{\zeta_{k}} T_{\xi_{k}} X_{k}$}
    \EndFor
    \State \Return $\{X_{k}\}_{k \in \mathbb{N}}$
  \end{algorithmic}
\end{algorithm}
\vspace*{-\baselineskip} Note, that this structure of two families of
mappings is a special case of \cref{algo:SSM}, just set $\tilde
T_{(i,j)} = S_{j} T_{i}$, where $(i,j) \in \tilde I = I \times J$ and
$\tilde \xi = (\xi,\zeta)$.  Also the Markov chain property is still
satisfied, the transition kernel takes the form $p(x,A) =
\mathbb{P}(S_{\zeta} T_{\xi}x \in A)$ for $x \in G$ and $A \in
\mathcal{B}(G)$. An advantage of this formulation is that properties
of the two families $\{S_{j}\}_{j \in J}$ and $\{T_{i}\}_{i \in I}$
can be analyzed more specifically, and independently.  As long as the
mapping $\tilde T$ satisfies the properties needed for convergence of
the RFI, then convergence of the RFI for two families of mappings
follows.  At the very least, we need
\[
C\equiv \mysetc{x \in G}{\mathbb{P}(x \in \Fix \tilde T_{\tilde\xi}) =
  1}\neq\emptyset.
\]
From this it is easy to see that for convergence the set $C_T$ {\em
  could} be empty, but the set $C_S$ {\em must} be nonempty.
\begin{example}[consistent feasibility]
  Revisit \cref{example:two_orth_halfspaces}. We had
  $C_{1}=\mathbb{R}_{+}\times \mathbb{R}$ and $C_{2}=\mathbb{R}\times
  \mathbb{R}_{+}$. Set $I=\{1\}$ and $J=\{2\}$, then the algorithm is
  the deterministic alternating projections method. One has
  $\mathbb{P}(X_{1} \in C) = 1$ for all initial distributions.
\end{example}

\begin{example}[inconsistent stochastic feasibility]\label{example:inconsist}
  In this example we show that the framework established here is not
  exclusively limited to consistent feasibility.  Consider the
  (trivially convex, nonempty and closed) set $S\equiv \{(0,10)\}$
  together with the collection of balls in \cref{example:disks},
  $C_{\alpha}:= \cb(\rho e_{\alpha},1) \subset \mathbb{R}^{2}$, where
  $0\leq\rho\leq 1$ and $e_{\alpha}=\icol{\cos(\alpha)\\
    \sin(\alpha)}$, $\alpha \in [0,2\pi)$ and let $\xi \sim
  \mathrm{unif}[0,2\pi]$.  The intersection of the disks is given by
  $C_T = \cb(0,1-\rho)$ where $T_{\alpha}\equiv P_{C_{\alpha}}$ is the
  metric projection onto $C_{\alpha}$.  Although $S \cap C_{\alpha} =
  \emptyset$ for all $\alpha \in [0,2\pi)$, still the fixed point set
  for the mapping in \cref{algo:SSM2} (where $S_{\zeta} = P_{S}$) is
  $C=\{(0,10)\}$, and this is found after one iteration for any
  initial probability distribution $\mu$, where $X_0\sim\mu$.

  This is indeed a special example, but points to the richness of
  inconsistent stochastic feasibility, which will be studied in
  greater depth in a follow-up paper.
\end{example}

\subsection{Linear Operator equations}
\label{sec:appl}

There are several applications of the RFI to the feasibility problem
\cite{ButnariuFlam1995}, \cite{Butnariu1995}. We want to focus first
on linear operator equations in the separable Hilbert space
$\mathcal{H} = L_{2}([a,b])$. Let
$\mymap{T}{\mathcal{H}}{\mathcal{H}}$ be a bounded linear operator, we
want to find $x \in \mathcal{H}$, such that
\begin{align*}
  T x = g,
\end{align*}
for a given $g \in \mathcal{H}$. Clearly this is possible only if $g
\in R(T)$. The idea in \cite{Butnariu1995} to solve $Tx=g$ is to consider the
family of evaluation mappings $\mymap{ \varphi_{t}}{ \mathcal{H}}{
  \mathbb{R}}$, $t \in [a,b]$, which are given by
\begin{align*}
  \varphi_{t}(x) := (Tx)(t).
\end{align*}
Define the affine subspaces $C_{t} := \mysetc{x \in \mathcal{H}}{
  \varphi_{t}(x) = g(t)}$, $t \in [a,b]$. Consider the probability
space $(\Omega,\mathcal{F},\mathbb{P}) = ([a,b], \mathcal{B}([a,b]),
\tfrac{\lambda}{b-a})$, where $\lambda$ is the Lebesgue-measure.  Let
$\mymap{\xi}{(\Omega, \mathcal{F}, \mathbb{P})}{([a,b],
  \mathcal{B}([a,b]))}$ be a random variable with
$\mathbb{P}^{\xi}=\mathbb{P}=\tfrac{\lambda}{b-a}$.  Then for $g \in
R(T)$, we have that
\begin{align*}
  Tx=g \qquad \text{if and only if} \qquad x \in C:=\mysetc{y \in
    \mathcal{H}}{\mathbb{P}(y \in C_{\xi})=1}.
\end{align*}
So the linear operator equation becomes a stochastic feasibility
problem.\\

In order to be able to compute projections onto the sets $C_{t}$, $t
\in [a,b]$, we need the evaluation functionals $\varphi_{t}$ to be
continuous, i.e.\ $\norm{\varphi_{t}} < \infty$ for almost all $t \in
[a,b]$. By the Riesz representation theorem there exists a unique
$u_{t} \in \mathcal{H}$ with $\varphi_{t}(x) = \act{u_{t},x}$ for all
$x \in \mathcal{H}$ and almost all $ t \in [a,b]$. We conclude that
the projection onto $C_{t}$ takes the form
\begin{align*}
  P_{t} x = x + \frac{g(t) - (Tx)(t)}{\norm{u_{t}}^{2}} u_{t} \qquad x
  \in L_{2}([a,b]).
\end{align*}

\begin{example}[linear integral equations]\label{example:lin_int_eq}
  Concretely, consider an integral equation of the first kind in the
  separable Hilbert space $L_{2}([a,b])$
  \begin{align*}
    (Tx)(t) = \int_{a}^{b} K(t,s) x(s) \dd{s} = g(t) \qquad t \in
    [a,b],
  \end{align*}
  with $g \in L_{2}([a,b])$.  For $K \in L_{2}([a,b]\times[a,b])$, $T$
  is a continuous linear compact operator \cite[Theorem
  8.15]{alt2002lineare} (a Hilbert-Schmidt operator).  For the Riesz
  representation of the evaluation functionals we have that
  $\varphi_{t}(x)=(Tx)(t) = \act{u_{t},x}$, $t \in [a,b]$, as well as
  $u_{t} = K(t,\cdot)$ and hence $\norm{\varphi_{t}} \le
  \norm{K(t,\cdot)}<\infty$.
\end{example}

\begin{example}[differentiation]
  Let $K(t,s) = \1_{[a,t]}(s) = u_{t}(s)$, i.e.\ $(Tx)(t) =
  \int_{a}^{t} x(s) \dd{s}$ and suppose $g \in C^{1}([a,b])$, then $Tx
  = g$ if and only if $ x = g'$ almost surely and $g(a)=0$. The
  projectors take the form
  \begin{align*}
    P_{t}x = x - \frac{g(t) - \int_{a}^{t}x(s)\dd{s}}{t-a} \1_{[a,t]}.
  \end{align*}
\end{example}

\appendix

\section{Appendix}
\label{sec:toolbox}

\begin{lemma}[slices of product $\sigma$-field, see Proposition 3.3.2
  in \cite{bogachev2007measure}] \label{lemma:slices_prod_field} Let
  $(\Omega_{i},\mathcal{F}_{i})$, $i=1,2$ be two measurable spaces and
  $M \in \mathcal{F}_{1} \otimes \mathcal{F}_{2}$. Then for
  $\omega_{1} \in \Omega_{1}$ holds $M_{\omega_{1}}:=
  \mysetc{\omega_{2}\in \Omega_{2}}{(\omega_{1},\omega_{2}) \in M} \in
  \mathcal{F}_{2}$.
\end{lemma}

\begin{thm}[dense sets in separable metric space]\label{thm:dense_sets_separable_space}
  Let $(G,d)$ be a Polish space (complete separable metric
  space). Then for any $A \subset G$, there is a dense countable
  subset $\{a_{n}\}_{n \in \mathbb{N}} \subset A$ and if $A$ is closed
  then even $A = \cl \{a_{n}\}$ ($\cl U$ denotes the closure of the
  set $U \subset G$ w.r.t.\ the metric $d$).
\end{thm}
\begin{proof}
  Since $G$ is separable there exists a dense and countable subset
  $\{u_{n}\}_{n \in \mathbb{N}} \subset G$ with $G = \cl
  \{u_{n}\}_{n}$. By denseness of $\{u_{n}\} \subset G$, for any $x
  \in G$ and any $\epsilon>0$, there is $u_{n}$, where $n$ is
  depending on $x$ and $\epsilon$, with $d(u_{n},x) < \epsilon$. Let
  $\epsilon >0$ and choose $a_{n}^{\epsilon} \in
  \mathbb{B}(u_{n},\epsilon) \cap A$, $n \in \mathbb{N}$, if the
  intersection is nonempty. The set $\tilde A := \{a_{n}^{1/m}\}_{n,m
    \in \mathbb{N}} \subset A$ is nonempty and countable as union of
  countable sets. It holds for any $a \in A$ and any $\epsilon > 0$
  that $\exists n,m$ with $1/m < \epsilon$ and $d(a,u_{n})<\epsilon$,
  hence
  \begin{align*}
    d(a,a_{n}^{1/m}) \le d(a,u_{n}) + d(u_{n},a_{n}^{1/m})
    < 2 \epsilon,
  \end{align*}
  i.e. $\tilde A \subset A$ dense. So then $A \subset \cl
  \tilde A$ and if $A$ is closed, then also $\cl \tilde A \subset A$.
\end{proof}

\begin{thm}[support of a measure]\label{thm:supp_measure}
  Let $(G, d)$ be a Polish space and $\mathcal{B}(G)$ its
  Borel $\sigma$-algebra. Let $\pi$ be a measure on
  $(G,\mathcal{B}(G))$ and define its support via
  \begin{align*}
    \supp\pi = \mysetc{x \in G}{\pi(\mathbb{B}(x,\epsilon)) > 0 \;
      \forall \epsilon >0}.
  \end{align*}
  Then
  \begin{enumerate}
  \item $\supp \pi \neq \emptyset$, if $\pi \neq 0$.
  \item $\supp \pi$ is closed.
  \item $\pi(A) = \pi(A \cap \supp \pi)$ for all $A \in
    \mathcal{B}(G)$, i.e. $\pi((\supp \pi)^{c}) = 0$.
  \item\label{item:supp_measure_closedsets} For any closed set $S
    \subset G$ with $\pi(A\cap S) = \pi(A)$ for all $A \in
    \mathcal{B}(G)$ holds $\supp \pi \subset S$.
  \end{enumerate}
\end{thm}
\begin{proof}
  \begin{enumerate}
  \item If $\pi(G)>0$, then due to separability one could find for any
    $\epsilon_{1} >0$ a countable cover of $G$ with balls
    with radius $\epsilon_{1}$, where at least one needs to have nonzero
    measure, because $0<\pi(G) \le \sum_{n} \pi(\mathbb{B}(x_{n}
    ,\epsilon))$. Now just consider $B_{1}:= \mathbb{B}(x_{N},\epsilon_{1})$ such that
    $\pi(B_{1})>0$ and apply the above procedure of countable covers
    with $\epsilon_{2}<\epsilon_{1}$ iteratively, then
    there is a sequence $\epsilon_{n} \to 0$ and $\mathbb{B}(x_{n+1},\epsilon_{n+1}) \subset
    \mathbb{B}(x_{n}, \epsilon_{n})$, such that $x_{n} \to x$, i.e. $x
    \in \supp \pi$.
  \item Let $(x_{n})_{n \in \mathbb{N}} \subset \supp \pi$ with $x_{n}
    \to x$ as $n \to \infty$. Let $\epsilon>0$ and $N >0$ such
    that $d(x_{n},x) < \epsilon$ for all $n \ge N$. Then $x_{n} \in
    \mathbb{B}(x,\epsilon)$ and $\exists \tilde \epsilon >0$ with
    $\mathbb{B}(x_{n}, \tilde\epsilon) \subset \mathbb{B}(x,
    \epsilon)$, so we get
    \begin{align*}
      \pi(\mathbb{B}(x, \epsilon)) \ge \pi(\mathbb{B}(x_{n},\tilde
      \epsilon)) > 0,
    \end{align*}
    i.e. $x \in \supp \pi$.
  \item Write $S=(\supp\pi)^{c}$. Choose $\{x_{n}\}_{n \in \mathbb{N}} \subset S$
    dense. By openness of $S$ there exists
    $\epsilon_{n}>0$ with $\mathbb{B}(x_{n}, \epsilon_{n}) \subset
    S$, hence $S = \bigcup_{n \in \mathbb{N}}
    \mathbb{B}(x_{n}, \epsilon_{n})$ and
    \begin{align*}
      \pi(S) \le \sum_{i\in\mathbb{N}}
      \pi(\mathbb{B}(x_{n},\epsilon_{n})) = 0.
    \end{align*}
    (It holds $\pi(\mathbb{B}(x_{n},\epsilon_{n}))=0$, because
    otherwise, one could find for any small enough $\epsilon >0$ a
    countable cover of $\mathbb{B}(x_{n},\epsilon_{n})$ with balls
    with radius $\epsilon$, where at least one needs to have nonzero
    measure. Since this holds for all $\epsilon$, there is a
    contradiction to $\mathbb{B}(x_{n},\epsilon_{n}) \subset S$.)
  \item Let $x\in \supp \pi$. So $\pi(\mathbb{B}(x,\epsilon) \cap S) >
    0$ for all $\epsilon >0$, i.e.\ $\mathbb{B}(x,\epsilon)\cap S \neq
    \emptyset$ for all $\epsilon>0$. Let $x_{n}$ be such that $x_{n} \in
    \mathbb{B}(x,\epsilon_{n}) \cap S$, where $\epsilon_{n} \to 0$ as
    $n \to \infty$. Then by closedness of $S$, $x_{n} \to x \in S$.
  \end{enumerate}
\end{proof}

\begin{thm}[construction of an invariant
  measure] \label{thm:construction_inv_meas} Let $(\mu
  \mathcal{P}^{n})_{n \in \mathbb{N}}$ be a tight family of
  probability measures on a Polish space $(G,d)$, i.e.\ for any
  $\epsilon>0$ there exists $K_{\epsilon} \subset G$ compact with
  $(\mu \mathcal{P}^{n})(G\setminus K_{\epsilon}) < \epsilon$ for all
  $n \in \mathbb{N}$, where $\mu \in \mathscr{P}(G)$ and $\mathcal{P}$
  is the Markov operator for a given transition kernel $p$, which is
  assumed to be Feller. Then any clusterpoint of $(\nu_{n})$ where
  $\nu_{n} = \tfrac{1}{n} \sum_{i=1}^{n} \mu\mathcal{P}^{i}$ is an
  invariant measure for $\mathcal{P}$.
\end{thm}
\begin{proof}
  This is basically \cite[Theorem 1.10]{Hairer2016}. The tightness of
  $(\mu \mathcal{P}^{n})$ implies tightness of $(\nu_{n})$ and
  therefore there exists a weakly converging subsequence
  $(\nu_{n_{k}})$ with limit $\pi \in \mathscr{P}(G)$ by Prokhorov's
  Theorem. By the Feller property of $\mathcal{P}$ one has for any
  continuous and bounded $\mymap{f}{G}{\mathbb{R}}$ that also
  $\mathcal{P}f$ is continuous and bounded and hence
  \begin{align*}
    \abs{(\pi \mathcal{P})f - \pi  f} &= \abs{\pi (\mathcal{P}f) - \pi
      f}  \\ &= \lim_{k} \abs{
    \nu_{n_{k}}(\mathcal{P} f)- \nu_{n_{k}} f} \\ &= \lim_{k}
  \frac{1}{n_{k}} \abs{ \mu\mathcal{P}^{n_{k}+1}f -
    \mu\mathcal{P}f} \\ &\le \lim_{k}
  \frac{2\norm{f}_{\infty}}{n_{k}} \\ &=0 .
  \end{align*}
\end{proof}

\begin{thm}[conditional expectation - basics, see Theorem 5.1 in
  \cite{kallenberg1997}]\label{thm:cond_exp_basics}
  Let $(\Omega,\mathcal{F},\mathbb{P})$ be a probability space and $X$
  a real-valued random variable with $\mathbb{E}\abs{X} < \infty$
  (integrable). Let $\mathcal{F}_{0} \subset \mathcal{F}$ a
  sub-$\sigma$-algebra. Then it exists an a.s.\ unique
  $\mathcal{F}_{0}$-mb.\ random variable $Z \equiv
  \cex{X}{\mathcal{F}_0}$ with $\mathbb{E}(Z\mathds{1}_{A}) =
  \mathbb{E}(X\mathds{1}_{A})$ for all $A \in \mathcal{F}_{0}$.\\
  Let $Y,X_{n}$ be integrable random variables. Further properties
  are:
  \begin{enumerate}
  \item $\mathbb{E}(\cex{X}{\mathcal{F}_{0}}) = \mathbb{E}X$.
  \item $X$ is $\mathcal{F}_{0}$-mb, then $\cex{X}{\mathcal{F}_{0}} =
    X$ a.s.
  \item $X$ independent of $\mathcal{F}_{0}$, then
    $\cex{X}{\mathcal{F}_{0}} = \mathbb{E} X$ a.s.
  \item\label{item:cex_basics:linear} $\cex{aX + b Y}{\mathcal{F}_{0}}
    = a\cex{X}{\mathcal{F}_{0}} + b \cex{Y}{\mathcal{F}_{0}}$ a.s.\
    for all $a,b\in \mathbb{R}$.
  \item\label{item:cex_basics:monotone} $X \le Y$, then
    $\cex{X}{\mathcal{F}_{0}} \le \cex{Y}{\mathcal{F}_{0}}$ a.s.
  \item $0\le X_{n} \uparrow X$, then $\cex{X_{n}}{\mathcal{F}_{0}}
    \uparrow \cex{X}{\mathcal{F}_{0}}$ a.s.
  \item $\mathcal{F}_{0}\subset\mathcal{F}_{1}\subset\mathcal{F}$ with
    $\sigma$-algebra $\mathcal{F}_{1}$, then
    $\cex{\cex{X}{\mathcal{F}_{1}}}{\mathcal{F}_{0}} = \cex{X}{\mathcal{F}_{0}}$.
  \end{enumerate}
\end{thm}

\begin{lemma}[Satz 17.11 in \cite{bauer1992}]\label{lemma:sigFintie}
  Let $(\Omega,\mathcal{F})$ be a measurable space and $\mu$ be
  $\sigma$-finite. Let $\mymap{f}{\Omega}{[0,\infty]}$ and set $\nu =
  f \cdot \mu$ (i.e.\ $\nu(A) = \int_{A} f \dd{\mu}$ for $A \in
  \mathcal{F}$). Then $f$ is $\mu$-a.s.\ unique. Furthermore, $\nu$ is
  $\sigma$-finite (i.e.\ there exists $(\Omega_{n})_{n \in \mathbb{N}}
  \subset \mathcal{F}$ with $\nu(\Omega_{n})< \infty$ and
  $\Omega_{n}\uparrow \Omega$) if and only if $f$ is real-valued
  $\mu$-a.s.
\end{lemma}
\begin{rem}\label{rem:NonnegVarIsSigFin}
  A nonnegative real-valued random variable $X$ on a probability space
  $(\Omega, \mathcal{F}, \mathbb{P})$ induces a $\sigma$-finite measure
  $\nu = X \cdot \mathbb{P}$.
\end{rem}

\begin{thm}[conditional expectation - extension]\label{thm:cond_expec}
  Let $(\Omega,\mathcal{F},\mathbb{P})$ be a probability space and
  $X\ge 0$ be a real-valued random variable. Let $\mathcal{F}_{0}
  \subset \mathcal{F}$ be a sub-$\sigma$-algebra. Then it exists an
  a.s.\ unique nonnegative real-valued random variable $Z \equiv
  \cex{X}{\mathcal{F}_0}$ on $(\Omega,\mathcal{F}_{0})$ with
  $\mathbb{E}(Z\mathds{1}_{A}) = \mathbb{E}(X\mathds{1}_{A})$ for all
  $A \in \mathcal{F}_{0}$.\\
  Let additionally $Y,(X_{n})$ be nonnegative and real-valued, then all items 1.-7.\ in
  \cref{thm:cond_exp_basics} are satisfied for these.
\end{thm}
\begin{proof}
  From \cref{rem:NonnegVarIsSigFin} follows the existence of disjoint
  sets $\Omega_{n} \in \mathcal{F}_{0}$ with $\bigcup_{n} \Omega_{n} =
  \Omega$ and the property that $\int_{\Omega_{n}} X \dd{\mathbb{P}} <
  \infty$. One has that a.s.
  \begin{align*}
    \1_{\Omega_{n}} \cex{X}{\mathcal{F}_{0}\cap \Omega_{n}}= \cex{ X
      \1_{\Omega_{n}}}{\mathcal{F}_{0} \cap \Omega_{n}}
    =\cex{X}{\mathcal{F}_{0}\cap \Omega_{n}} = \cex{X
      \1_{\Omega_{n}}}{\mathbb{F}_{0}}.
  \end{align*}
  Define $Z := \sum_{n} \cex{X}{\mathcal{F}_{0}\cap \Omega_{n}}$, then
  $Z = \cex{X}{\mathcal{F}_{0}}$. The items 1.-7.\ follow now from
  \cref{thm:cond_exp_basics} on $\Omega_{n}$ and the Monotone
  Convergence Theorem.
\end{proof}

\begin{thm}[disintegration]\label{thm:disinteg}
  Let $(\Omega,\mathcal{F},\mathbb{P})$ be a probability space and
  $(\Omega_{1},\mathcal{F}_{1})$, $(\Omega_{2}, \mathcal{F}_{2})$ be
  measurable spaces and $\mathcal{F}_{0} \subset \mathcal{F}$ a
  sub-$\sigma$-algebra, let $X_{1} \in \Omega_{1}$ have a regular
  version $\mu$ of $\cpr{X_{1} \in \cdot}{\mathcal{F}_{0}}$ and let
  $X_{2} \in \Omega_{2}$ be $\mathcal{F}_{0}$-mb. Let furthermore
  $\mymap{f}{\Omega_{1} \times \Omega_{2}}{\mathbb{R}_{+}}$ be
  measurable. Then
  \begin{align*}
    \cex{f(X_{1},X_{2})}{\mathcal{F}_{0}} = \int f(x_{1},X_{2})
    \mu(\dd{x_{1}}) \quad \text{a.s.}
  \end{align*}
\end{thm}

\begin{lemma} \label{lemma:supermartingale} Let
  $(\Omega,\mathcal{F},\mathbb{P})$ be probability space. Let
  $(X_{k})_{k \in \mathbb{N}_{0}}$, $(U_{k})_{k\in \mathbb{N}_{0}}$ be
  sequences of nonnegative real-valued random variables with $X_{k}
  \in \mathcal{F}_{k}$, where $\mathcal{F}_{0} \subset \mathcal{F}_{1}
  \subset \ldots \subset \mathcal{F}$ are $\sigma$-algebras. Suppose
  for all $k \in \mathbb{N}_{0}$
  \begin{align*}
    X_{k+1} \le X_{k} - U_{k} \quad a.s.
  \end{align*}
  Define $V_{k}:= \cex{U_{k}}{\mathcal{F}_{k}}$ for $k\in
  \mathbb{N}_{0}$. Then $X_{k} \to X$ a.s.\ and $\sum_{k} U_{k},
  \sum_{k} V_{k} < \infty$ a.s.
\end{lemma}
\begin{proof}
  This is a special instance of the more general supermartingale
  convergence theorem in \cite{ROBBINS1971233}.
\end{proof}

\begin{lemma}[further properties of $R$]\label{lemma:further_prop_R}
  Under the standing assumptions, if $T_{i} = P_{i}$ are projectors
  onto nonempty, closed and convex sets, $i \in I$. There holds that:
  \begin{enumerate}
  \item $R$ is convex.
  \item\label{item:prop_R:derivative} $R$ is continuously
    differentiable, $\tfrac{1}{2}\nabla R(x) = x -
    \mathbb{E}[P_{\xi}x]$ for all $x \in \mathcal{H}$.
  \item $\nabla R$ is globally Lipschitz continuous with constant not
    larger than $4$.
  \item $C=\{\nabla R=0\}$.
  \end{enumerate}
\end{lemma}
\begin{proof}
  \begin{enumerate}
  \item The function $x\mapsto \dist(x,C_{i})$ is convex for all $i
    \in I$, since $C_{i} = \Fix P_{i}$ is convex, nonempty and
    closed. On $[0,\infty)$ the function $x\mapsto x^{2}$ is
    increasing and convex, so $x\mapsto \dist^{2}(x,C_{i})$ is convex,
    $i \in I$. The convexity of $R$ follows by linearity of the
    expectation.
  \item We need to show that
    \begin{align*}
      \lim_{0 \neq \norm{y} \to 0} \frac{\abs{ R(x+y) - R(x) -
          2\mathbb{E}[\act{x-P_{\xi}x,y}] }}{\norm{y}} =0.
    \end{align*}
    Let $(y_{n}) \subset \cb(0,\epsilon) \subset \mathcal{H}$ with
    $y_{n} \to 0$. Define a sequence of functions on $\Omega$ via
    \begin{align*}
      f_{n} = \frac{\abs{ \dist^{2}(x+y_{n},C_{\xi}) -
          \dist^{2}(x,C_{\xi}) - 2\act{x-P_{\xi}x,y_{n}}
        }}{\norm{y_{n}}}.
    \end{align*}
    Then for fixed $\omega \in \Omega$ we have $f_{n}(\omega) \to 0$
    as $n \to \infty$ since the function $x \mapsto
    \dist^{2}(x,C_{i})$ is Fr\'echet differentiable for all $i \in I$
    \cite[Corollary 12.30]{Bauschke2011}. Furthermore, we find for any
    $n \in \mathbb{N}$ that
    \begin{align*}
      f_{n} &= \frac{\abs{ \norm{x+y_{n} - P_{\xi}(x+y_{n})}^{2} -
          \norm{x-P_{\xi}x}^{2} - 2\act{x-P_{\xi}x,y_{n}}
        }}{\norm{y_{n}}} \\
      &= \frac{\abs{ \norm{y_{n} - P_{\xi}(x+y_{n}) + P_{\xi}x}^{2} +
          2\act{x-P_{\xi}x,P_{\xi} x - P_{\xi}(x+y_{n})}
        }}{\norm{y_{n}}} \\
      &= \frac{\abs{ \norm{y_{n}}^{2} + \norm{P_{\xi}x -
            P_{\xi}(x+y_{n})}^{2} + 2\act{y_{n}+x-P_{\xi}x,P_{\xi} x -
            P_{\xi}(x+y_{n})} }}{\norm{y_{n}}} \\
      &\le \frac{ 4\norm{y_{n}}^{2}+ 2\dist(x,C_{\xi}) \norm{y_{n}} }{\norm{y_{n}}} \\
    &\le 4\epsilon + 2\dist(x,C_{\xi}) =: g,
  \end{align*}
  where, in the first inequality, we used nonexpansivity of the
  projectors $P_{i}$, $i \in I$ and the Cauchy-Schwartz inequality.
  In particular with H\"older's inequality follows that $\mathbb{E}[g]
  \le 4\epsilon + 2 \sqrt{R(x)}$, i.e.\ $g$ is integrable and hence
  Lebesgue's Dominated Convergence Theorem yields $\mathbb{E}[f_{n}]
  \to 0$, which gives us Fr\'echet differentiability of $R$ with
  derivative $\nabla R(x)= 2\mathbb{E}[x-P_{\xi}x]$. Continuity of
  $\nabla R$ follows from
  \begin{align*}
    \norm{\nabla R(x+y) - \nabla R(x)} &=
    2\norm{\mathbb{E}[y-P_{\xi}(x+y)+P_{\xi}x]} \\ &\le 2
    \mathbb{E}[ \norm{y}+\norm{P_{\xi}(x+y)-P_{\xi}x} ] \\ &\le 4
    \norm{y},
  \end{align*}
  where we used nonexpansivity of the projectors $P_{i}$, $i \in I$
  in the second inequality.
  \item For any $x,y\in \mathcal{H}$ it holds that $\norm{\nabla R(x)-
      \nabla R(y)} \le 2 ( \norm{x-y} + \norm{\mathbb{E}[P_{\xi}x -
      P_{\xi}y]})$. Applying Jensen's inequality and nonexpansivity,
    we arrive at the desired result.
  \item Clearly if $x \in C$, then $x=P_{\xi}x$ a.s.\ and so
    $x=\mathbb{E}[P_{\xi}x]$, i.e.\ $\nabla R(x)=0$ by
    \ref{item:prop_R:derivative}.\\
    Now conversely, if $\nabla R(x)=0$, then by convexity
    $R(x)-R(y)\le \act{\nabla R(x),x-y}=0$ for all $y \in
    \mathcal{H}$. Since $C\neq \emptyset$ there is $y\in \mathcal{H}$
    with $R(y)=0$, so also $R(x)=0$, i.e.\ $x \in C$.
  \end{enumerate}
\end{proof}

\section{Paracontractions}
\label{sec:paracont}

Paracontractions include the set of averaged operators, but averaged
mappings possess more useful regularity properties, e.g.\ when
composing these operators, one stays in the set of averaged operators,
whereas for nonaveraged operators this is not clear in general. An
example of a nonaveraged paracontraction in $\mathbb{R}$ is a Huber
function with parameter $\alpha>0$ (see also \cite[Example
2.3]{BausckeBorwein96} for $\alpha=1$)
\begin{align*}
  f_{\alpha}(x) :=
  \begin{cases}
    \frac{x^{2}}{2\alpha},& \abs{x}\le \alpha\\
    \abs{x}-\frac{\alpha}{2},& \abs{x}>\alpha\\
  \end{cases}, \qquad
  x \in \mathbb{R}.
\end{align*}
We have that $f_{\alpha}$ is nonexpansive and paracontractive, but not
averaged, since for $x=-2\alpha$ and $y=-\alpha$ one has
$f(x)=\tfrac{3\alpha}{2}$ and $f(y) = \tfrac{\alpha}{2}$.  Consequently
\begin{align*}
  \abs{f(x)-f(y)} = \alpha = \abs{x-y}, \qquad \text{but} \quad \abs{x-f(x)-(y-f(y))} =
  2\alpha \neq 0.
\end{align*}

In general metric spaces with nonlinear structure the averaged
mappings are not defined or at least demand a different definition,
but still the paracontraction framework applies here and exhibits a
useful description of mappings which ensure that the RFI converges to a
common fixed point. Paracontractions were used in \cref{sec:compact metr} to 
guarantee Fej\'er monotonicity, yielding convergence; averagedness 
in this context would be too strong an assumption
(and is not defined actually). In the
following we provide an example of a class of paracontracting
operators in $\mathbb{R}^{n}$, that are not in general averaged,
resolvents of quasiconvex functions.

\begin{definition}
  A function $\mymap{f}{\mathbb{R}^{n}}{\mathbb{R}}$ is called
  quasiconvex, if the sublevel sets
  \begin{align*}
    \mysetc{x \in
      \mathbb{R}^{n}}{f(x) \le \alpha}
  \end{align*}
  are convex for all $\alpha \in \mathbb{R}$. Equivalently, $f$
  satisfies
  \begin{align*}
    f(\lambda x + (1-\lambda)y) \le \max\{f(x),f(y)\} \qquad \forall
    x,y \in \mathbb{R}^{n}, \, \forall \lambda \in [0,1].
  \end{align*}
\end{definition}

The proximity operator function
$\mymap{f}{\mathbb{R}^{n}}{\mathbb{R}}$ is given by the set-valued
mapping
\begin{align*}
  \prox_{f}(x) := \argmin_{y \in \mathbb{R}^{n}} \left\{ f(y) +
    \frac{1}{2}\norm{x-y}^{2} \right\}, \qquad x \in \mathbb{R}^{n}.
\end{align*}
\begin{lemma}\label{lemma:quasiconvexParacontraction}
  Let $\mymap{f}{\mathbb{R}^{n}}{\mathbb{R}}$ be twice differentiable
  and quasiconvex and satisfy $S:=\argmin f \neq \emptyset$ and
  $\nabla f \neq 0$ on $\mathbb{R}^{n} \setminus S$, furthermore
  suppose that $\id + \Hess f(x)$ is positive definit for all $x \in
  \mathbb{R}^{n}$, then $\prox_{f}$ is paracontracting.
\end{lemma}
\begin{proof}
  Denote $A:= \id+\nabla f$. Let $x,y \in \mathbb{R}^{n}$ with $f(x)
  \ge f(y)$, then
  \begin{align*}
    \norm{A(x) - y}^{2} = \norm{x - y}^{2} + \norm{\nabla
      f(x)}^{2} + \act{ \nabla f(x), x - y} \ge \norm{ x - y} ^{2},
  \end{align*}
  where we used that in \cite{zbMATH03169925} it is shown, that a
  quasiconvex and differentiable function satisfies
  \begin{align*}
    f(x) \ge f(y) \qquad \Longrightarrow \qquad \act{\nabla f(x), x - y} \ge 0,
  \end{align*}
  for any $x,y \in \mathbb{R}^{n}$. Note that if $x \notin S$ then
  $\nabla f(x) \neq 0$ by assumption and hence for $y \in
  \mathbb{R}^{n}$ with $f(y) \le f(x)$ it holds that
  \begin{align}\label{eq:quasicvx_ineq1}
    \norm{A(x) - y} > \norm{x-y}.
  \end{align}
  Moreover, the function
  \begin{align*}
    g(y) := f(y) + \frac{1}{2} \norm{x-y}^{2}
  \end{align*}
  for fixed $x \in \mathbb{R}^{n}$ is bounded from below, since
  $\inf_{x} f(x) > -\infty$ by assumption and coercive. From positive
  definitness of $\id+\Hess f$ we have that $g$ is also twice
  continuously differentiable and strictly convex, hence it possesses
  a unique minimizer $\bar x$ that satisfies
  \begin{align*}
    x = \nabla f(\bar x) + \bar x = A(\bar x),
  \end{align*}
  it follows that $A(\mathbb{R}^{n}) = \mathbb{R}^{n}$, i.e.\ $A$ is
  surjective. Furthermore, $A$ is injective, since from uniqueness of
  the minimizer and sufficiency of the first order optimality
  criterion for $\bar x$ to be a minimizer ($g$ is convex) it follows
  that, if $A(\bar x) = A(\bar y)$, then $\bar x = \bar y$ is the
  minimizer for $g$ and in particular $A(\bar x) =x \, \Leftrightarrow
  \, \prox_{f}(x) = \bar x$.

  To show that also $\prox_{f}$ is continuous, fix $x \in \mathbb{R}^{n}$ and
  let $y \in \mathbb{B}(x,\epsilon)$. We can find a $z \in
  \cb(x,\epsilon)$ with $f(z) \le f(y)$ for all $y \in
  \cb(x,\epsilon)$ by continuity of $f$, so we get with
  \eqref{eq:quasicvx_ineq1} that
  \begin{align*}
    \norm{\prox_{f}(x)-\prox_{f}(y)} \le \norm{\prox_{f}(x)-z} +
    \norm{\prox_{f}(y) - z} < \norm{x-z} + \norm{y-z} < 2\epsilon.
  \end{align*}
  In particular, letting $y=\bar x \in S$ in
  \eqref{eq:quasicvx_ineq1}, we have that
  \begin{align*}
    \norm{\prox_{f}(x) - \bar x} < \norm{x - \bar x} \qquad \forall x \in
    \mathbb{R}^{n} \setminus S,
  \end{align*}
  where $S = \argmin f = \Fix \prox_{f}$.
\end{proof}

\begin{example}[non-averaged resolvent of quasiconvex function]
  The function $f(x) := 1-\exp(-\norm{x}^{2})$ for $x \in
  \mathbb{R}^{n}$ satisfies all the conditions in
  \cref{lemma:quasiconvexParacontraction}. Its proximity operator has
  the derivative $\prox_{f}'(A(x)) = (A'(x))^{-1}$, where $A(x) =
  \left(1+2\exp(-\norm{x}^{2})\right)x$, i.e.\ $A'(x) =
  \left(1+2\exp(-\norm{x}^{2})\right)\id - 4 \exp(-\norm{x}^{2}) x
  x^{\mathsf{T}}$. Since $\norm{\prox_{f}'(A(x))} \ge
  \norm{y}/\norm{A'(x) y}$ for any $y \in \mathbb{R}^{n} \setminus
  \{0\}$, we have with $x=e_{1}=(1,0,\ldots,0)^{\mathsf{T}}=y$ that
  $\norm{\prox_{f}'(A(e_{1}))} > 1$, which is in contradiction to
  nonexpansiveness of averaged mappings, that have derivative bounded
  by $1$, if it exists.
\end{example}

Where paracontractions also occur are nonconvex feasibility problems,
both consistent and inconsistent. As long as the fixed point set of
the averaged projections operator consists of isolated points and the projectors
are single-valued in a neighborhood of this fixed point, \cite[Theorem
3.2]{Luke2016a} shows that these operators are paracontractions,
whenever all assumptions of the theorem are met. Unfortunately, a
statement on paracontractiveness, for the case that the fixed point
set of the averaged projections operator does not consist of isolated points, is
however not possible in general.

Furthermore, also nonconvex forward-backward operators appearing in
structured optimization of nonconvex objective functions show the
paracontractiveness property, see \cite[Proposition 3.9]{Luke2016a},
and these are not averaged in general, still the assumption that the 
fixed points are isolated is used.

\bibliography{lit}

\begin{thebibliography}{10}

\bibitem{ButnariuFlam1995}
D.~Butnariu and S.~D. Fl\r{a}m, ``{Strong convergence of expected-projection
  methods in hilbert spaces},'' {\em Numerical Functional Analysis and
  Optimization}, vol.~16, no.~5-6, pp.~601--636, 1995.

\bibitem{Nedic2011}
A.~Nedi{\'{c}}, ``{Random algorithms for convex minimization problems},'' {\em
  Mathematical Programming}, vol.~129, no.~2, pp.~225--253, 2011.

\bibitem{flam1995}
S.~D. Flåm, ``Successive averages of firmly nonexpansive mappings,'' {\em
  Mathematics of Operations Research}, vol.~20, no.~2, pp.~497--512, 1995.

\bibitem{Butnariu1995}
D.~Butnariu, ``{The expected-projection method: Its behavior and applications
  to linear operator equations and convex optimization},'' {\em Journal of
  Applied Analysis}, vol.~1, no.~1, pp.~93--108, 1995.

\bibitem{Nedic2010}
A.~Nedi{\'{c}}, ``{Random projection algorithms for convex set intersection
  problems},'' {\em 49th IEEE Conference on Decision and Control},
  pp.~7655--7660, 2010.

\bibitem{Luke2016b}
D.~R. Luke, M.~Teboulle, and N.~H. Thao, ``{Necessary conditions for linear
  convergence of Picard iterations and application to alternating
  projections},'' {\em arXiv}, 2017.

\bibitem{kallenberg1997}
O.~Kallenberg, {\em Foundations of Modern Probability}.
\newblock Probability and Its Applications, New York: Springer, 1997.

\bibitem{Bauschke2011}
H.~H. Bauschke and P.~L. Combettes, {\em {Convex Analysis and Monotone Operator
  Theory in Hilbert Spaces}}.
\newblock Berlin: Springer, 2011.

\bibitem{mann1953mean}
W.~R. Mann, ``Mean value methods in iterations,'' {\em Proc. Amer. Math. Soc.},
  vol.~4, pp.~506--510, 1953.

\bibitem{krasnoselski1955}
M.~A. Krasnoselski, ``Two remarks on the method of successive approximations,''
  {\em Math. Nauk. (N.S.)}, vol.~63, no.~1, pp.~123--127, 1955.
\newblock (Russian).

\bibitem{edelstein1966}
M.~Edelstein, ``{A remark on a theorem of M. A. Krasnoselski},'' {\em Amer.
  Math. Monthly}, vol.~73, pp.~509--510, May 1966.

\bibitem{Gubin67}
L.~Gubin, B.~Polyak, and E.~Raik, ``The method of projections for finding the
  common point of convex sets,'' {\em USSR Comput. Math. and Math. Phys.},
  vol.~7, no.~6, pp.~1--24, 1967.

\bibitem{BaiBruRei78}
J.~B. Baillon, R.~E. Bruck, and S.~Reich, ``On the asymptotic behavior of
  nonexpansive mappings and semigroups in {B}anach spaces,'' {\em Houston J.
  Math.}, vol.~4, no.~1, pp.~1--9, 1978.

\bibitem{Hairer2006}
M.~Hairer, ``{Ergodic properties of Markov processes},'' {\em Lecture Notes in
  Mathematics}, vol.~1881, pp.~1--39, 2006.

\bibitem{Deutsch2001}
F.~Deutsch, {\em Best Approximation in Inner Product Spaces}.
\newblock New York: Springer, 2001.

\bibitem{Kruger2016}
A.~Y. Kruger, D.~R. Luke, and N.~H. Thao, ``{Set regularities and feasibility
  problems},'' {\em Mathematical Programming}, vol.~168, no.~1-2, pp.~279--311,
  2018.

\bibitem{Bolte2017}
J.~Bolte, T.~P. Nguyen, J.~Peypouquet, and B.~W. Suter, ``From error bounds to
  the complexity of first-order descent methods for convex functions,'' {\em
  Mathematical Programming}, vol.~165, pp.~471--507, Oct 2017.

\bibitem{alt2002lineare}
H.~Alt, {\em {Lineare Funktionalanalysis: Eine anwendungsorientierte
  Einf{\"u}hrung (in German)}}.
\newblock Springer Lehrbuch, Berlin: Springer, 2002.

\bibitem{bogachev2007measure}
V.~Bogachev, {\em Measure Theory}, vol.~I.
\newblock Springer Berlin Heidelberg, 2007.

\bibitem{Hairer2016}
M.~Hairer, ``{Convergence of Markov processes},'' {\em Lecture notes,
  University of Warwick}, 2016.

\bibitem{bauer1992}
H.~Bauer, {\em {Ma\ss- und Integrationstheorie (in German)}}.
\newblock De Gruyter Lehrbuch, Berlin: W. de Gruyter, 1992.

\bibitem{ROBBINS1971233}
H.~Robbins and D.~Siegmund, ``A convergence theorem for non negative almost
  supermartingales and some applications,'' in {\em Optimizing Methods in
  Statistics} (J.~S. Rustagi, ed.), pp.~233 -- 257, Academic Press, 1971.

\bibitem{BausckeBorwein96}
H.~H. {Bauschke} and J.~M. {Borwein}, ``{On projection algorithms for solving
  convex feasibility problems.},'' {\em {SIAM Rev.}}, vol.~38, no.~3,
  pp.~367--426, 1996.

\bibitem{zbMATH03169925}
K.~J. {Arrow} and A.~C. {Enthoven}, ``{Quasi-concave programming.},'' {\em
  {Econometrica}}, vol.~29, pp.~779--800, 1961.

\bibitem{Luke2016a}
D.~R. Luke, N.~H. Thao, and M.~K. Tam, ``{Quantitative convergence analysis of
  iterated expansive , set-valued mappings},'' {\em Mathematics of Operations
  Research}, pp.~1--33, 2018.

\end{thebibliography}

\end{document}